\documentclass[12pt]{article}
\usepackage{amsmath,amsfonts,amsthm,amssymb,amscd}
\usepackage[toc,page]{appendix}

\numberwithin{equation}{section}

\newcommand{\be}{\begin{equation}}
\newcommand{\ee}{\end{equation}}

\newcommand{\bs}{\begin{split}}
\newcommand{\es}{\end{split}}
\newcommand{\ba}{\begin{align}}
\newcommand{\ea}{\end{align}}
\newcommand{\basl}[1]{\begin{align}\begin{split}\label{#1}}
\newcommand{\bas}{\begin{align}\begin{split}}
\newcommand{\bl}[1]{\begin{equation}\label{#1}}
\newcommand{\el}[1]{\end{equation}\label{#1}}

\newtheorem{theo}{Theorem}[section]
\newtheorem{prop}[theo]{Proposition}

\newtheorem{defi}[theo]{Definition}

\newcommand\fpr{\hfill$\Box$\null}

\newcommand\N{\mathbb{N}}

\newcommand\R{\mathbb{R}}
\newcommand\C{\mathbb{C}}

\newcommand\E{\mathbb{E}}

\title{ Quantization and process}

\author{L. Amour\textsuperscript{1}, R. Lascar\textsuperscript{2} and
J. Nourrigat\textsuperscript{1}}

\date{ \ \textsuperscript{1}LMR - UMR CNRS 9008, Universit\'e de Reims Champagne-Ardenne, France\hskip 0cm \ \\ \vskip 0.3cm
\textsuperscript{2} LJAD - UMR 735, Universit\'e C\^ote d'Azur, Parc Valrose,  France
}

\begin{document}

\maketitle

\begin{abstract}
\noindent
This article is concerned with generalizations of pseudo-differential operators in $L^2(\R^n)$, $n\geq 1$. The definition of the new calculi  depends only on  bounded measures on the phase space $\R^{2n}$ and each measure gives rise to a specific calculus. Quantizations  of anti-Wick, Weyl, classical and Born-Jordan are  particular cases of the general calculi. Classes of symbols in this framework are then studied. The Gevrey class of parameter 1/2 is a class of symbol that is common to all the general calculi, that is, a class of symbols independent on  the bounded measures parametrizing the quantizations.  Precise additional hypotheses on the measures are necessary in the aim to consider the  larger class of symbols  $L^{\infty}(\R^{2n})$. This result can be applied for anti-Wick but not for Weyl quantization.  Concerning    Weyl pseudo-differential calculus, we recover the standard class of Sj\"ostrand and Gr\"ochening. Then, we prove  that  probability measures of L\'evy processes on the phase space   $\R^{2n}$ with diffusion larger than 1/4 are natural examples of measures satisfying the latter additional hypotheses in order to consider $L^{\infty}(\R^{2n})$ symbols.  This relation is derived using the L\'evy-Khintchine formula. Composition laws in that general context are next investigated. In that purpose, we give a formula for the composition of two symbols in some precise class of symbols valid for all general quantizations.  General calculi are relying on Wick  quantization which is therefore primarily examined for some precise classes of symbols. Additional results in that context are provided, such as Mizrahi series  expansions  and Banach algebra isomorphisms between operators and symbol classes.
\end{abstract}

\parindent=0pt

{\it Keywords}: Pseudo-differential operators, composition laws of symbols, Levy processes, L\'evy-Khintchine representation, Mizrahi series, Weyl quantization, anti-Wick quantization, Wick quantization, classical quantization, Born-Jordan quantization, isomorphisms of algebras, Berezin symbol.


\

{\it MSC 2010:} 	35S05, 47B33, 47G30.

\tableofcontents

\parindent=0pt

\parskip 3pt
\baselineskip 15pt 

\section{Introduction. }

Pseudo-differential  calculi, including  Wick, anti-Wick, Weyl, classical (or left) and Born-Jordan quantizations,
 assign to some function $F$ on $\R^{2n}$  (called  {\it symbol}) a bounded operator in   $L^2(\R^n)$ depending on a parameter $h>0$. Wick and  Weyl  calculi are recalled in Section 2. Anti-Wick, classical and Born-Jordan calculi are reminded in  Section 6. See also, e.g.,  \cite{HO} and  \cite{LER}. 
 
The purpose of  in this article is to prove that one can actually construct an infinite number of other  pseudo-differential calculi. The formula of the generalized calculi depends only on a measure and one recovers  give the usual calculi (anti-Wick, Weyl, classical and Born-Jordan) with particular choices of the measure.

In that aim, we firstly examine  Wick calculus (Section \ref{Wick calculus}).  For each  $h>0$, we define two Banach algebras, namely,  a Banach  algebra $C_h$ of bounded operators in  
 $L^2 (\R^{n})$ and   a Banach algebra $D_h (\R^{2n})$ of functions defined on the phase space  $\R^{2n}$. We show that the mapping  $A\rightarrow \sigma^{\rm Wick}(A) $, where  $\sigma^{\rm Wick}(A) $ denotes the Wick symbol of an operator $A$,  is an  isomorphism between the two Banach algebras  $C_h$ and $D_h$, for every $h>0$. 
 The operator assigned to a given symbol  $F$ with the Wick calculus is given by formula   (\ref{def-A-Wick}), which may look unfamiliar  to define a pseudo-differential operator.  
The union on all $h>0$ of the classes of symbols  $D_h (\R^{2n})$ is  the  Gevrey space of parameter 1/2. 

 Secondly, we  notice a connection  between usual pseudo-differential calculi  (anti-Wick, Weyl, classical, Born-Jordan) with the Wick calculus. More precisely,  the  Weyl, 
   anti-Wick, classical and  Born-Jordan 
 quantizations of a symbol $F$ can be defined using the Wick quantization by   applying  to $F$ a  convolution 
   by a  measure depending on the considered quantization.
   The four measures corresponding   to these usual quantizations are written down in  (\ref{mesure-Weyl})-(\ref{mesure-BJ}). 
 
 Therefore, we can extend this procedure by considering  any bounded measures $\nu$  on the phase space 
  $\R^{2n}$ instead of choosing only the four measures corresponding to the standard quantizations. Namely, with a given  bounded measure $\nu$ and  a symbol $F$ in the space $D_h (\R^{2n})$, we define  the operator  ${\rm Op}_h^{\nu}(F)$, the quantization of $F$  for the calculus indexed by $\nu$, using the Wick calculus of the convolution product  $\nu \star F$.

This quantized operator is precisely defined in   (\ref{def-gene}). It is well defined since the  convolution by a bounded measure maps $D_h (\R^{2n})$  
 into itself. Formula  (\ref{def-gene}) for bounded measure $\nu$ is thus highly generic since its can be applied with most of the usual known pseudo-differential calculi. This can be effectuated only for symbols in  $D_h (\R^{2n})$ and therefore in  the Gevrey space of parameter 1/2.  We underline at this stage that, $D_h (\R^{2n})$ or Gevrey space of parameter 1/2, is always a possible common  space of symbols for {\it all} the general calculi.
  
Then, in the purpose to quantize more general symbols than elements of $D_h (\R^{2n})$, we do not consider a chosen single measure but rather  a convolution semi-group of measures  $(\nu_h)$. More precisely, we define  a set  ${\cal M} ( \gamma , \R^{2n} )$  ($\gamma <0$) 
          of family of measures (see Definition  (\ref{def-op-D})).   The convolution by these measures has a regularizing effect. Moreover, if  $4 |\gamma | >1$ then the convolution maps $L^{\infty} (\R^{2n})$  
         into  $D_h (\R^{2n})$. The upshot is that  ${\rm Op}_h^{\nu_h}(F)$
          can be defined by formula (\ref{def-gene})  for all $F$ in 
          $L^{\infty} (\R^{2n})$ instead of $D_h (\R^{2n})$, for any  family of measures in ${\cal M} ( \gamma , \R^{2n} )$ with $4 |\gamma | >1$. 
          
 This method can be applied to the anti-Wick calculus in order to consider symbols in $L^{\infty} (\R^{2n})$.

For $L^{\infty} (\R^{2n})$ symbols, we also consider Levy processes.
         If $(X_t)$ is a  L\'evy process on  $\R^{n}$ with definite positive 
             diffusion matrix then the family  
           $(\nu_t)$ of transition probabilities belongs to a
  set  ${\cal M} ( \gamma , \R^{n} )$. This proves that these transition probability densities have a holomorphic extensions satisfying (\ref{densit-trans}) which give a precision of a result in 
             \cite{KN0-SCH}. If we  now  consider a  L\'evy process on the phase space $\R^{2n}$ then we can use the family of  transition probabilities in order to define a pseudo-differential calculi 
             for symbol in  $L^{\infty} (\R^{2n})$ using formula  (\ref{def-gene}). This relies on the L\'evy-Khintchine formula.

Nevertheless, one notes that 
  the set of measures ${\cal M} ( \gamma , \R^{2n} )$ cannot be used for all pseudo-differential calculi since it is already known that  $L^{\infty} (\R^{2n})$ is not always a suitable class of symbols to get bounded operators. For the Weyl calculus,
we use the class of symbols $S_W$  defined in Theorem 
 \ref{S-Weyl} which is identical to a class of symbols of   Gr\"ochenig  \cite{GRO} and very closely related to the class of Sj\"ostrand  \cite{SJ-1}. We prove in particular that the 
 convolution by the measure corresponding to the  Weyl calculus maps $S_W$ into $D_h (\R^{2n})$ allowing to define the Weyl quantization by (\ref{def-gene}).     
     Note that we do not have have similar results for the   Born-Jordan quantization and we can only therefore quantize symbols in 
$D_h (\R^{2n})$.

We next  turn  to the study of the composition of two operators being  quantizations with a general calculi  of two symbols  $F$ and  $G$. The results below  apply   for  $F$ and  $G$ in some class $D_{\lambda h} (\R^{2n})$.
The composition is written under two  forms for the Wick calculus and  one form for general calculi.  
For the Wick calculus, the first composition formula (\ref{sharp-Wick}) expresses the symbol of the composition of 
 $F$ and  $G$ . It is  written not directly in terms of  $F$ and $G$ but rather with some extensions of $F$ and $G$. Note that in this formula, we suppose that  $F$ and $G$ belong to $D_{ h} (\R^{2n})$ without any additional hypotheses. The second composition formula (\ref{comp-Miz}) for the Wick calculus gives   the symbol of the composition directly with $F$ and $G$ using a  Mizrahi  series. The series is absolutely converging only with supplementary assumptions 
on  $F$ and
$G$.  

For the general quantizations with  families of measures in 
 ${\cal M} ( \gamma , \R^{2n} )$ and for symbols  
   $F$ and $G$ belonging to $D_{\lambda h} (\R^{2n})$ ($\lambda$ large enough), we write the composition of the symbols in three steps, namely:

 $\quad \bullet$ Consider the tensor product
   $ (F \otimes G) (X , Y) = F(X) G(Y) .$
   
    $\quad \bullet$ Apply a convolution operator (with variables in  $\R^{4n}$)  to $F \otimes G$.    
    
     $\quad \bullet$ Restrict the result  to   the diagonal  $X=Y$. 
   
 As an example, one usually considers that the composition of two anti-Wick operator is not an anti-Wick operator (unless if one regards the symbol of the composition as  a Gelfand Shilov generalized function, see \cite{ALN}, and for general properties of Gelfand-Shilov generalized functions, see \cite{N-R}). However, if the symbols belong to one of the classes of Section \ref{Wick calculus} then the symbol of the composition belongs to another of these classes. 
 
Using the above general composition formula, we  recover composition  formulas  for the usual quantizations. These usual composition formulas are more often written as asymptotic expansion and are non exact.    

When the densities of the measures are exponential functions of quadratic forms then the above general composition formulas can be written as series. If one chooses to work in larger spaces than Gevrey 1/2 then one could expect  obtaining series that are also asymptotic (see \cite{LER}, Theorem 2.4.6). Moreover, one notices that these associative composition laws defined by series  remind star-products of    Boutet de Monvel  \cite{BM-1} and \cite{BM-2}. 
      
Besides, let us make the following supplementary observation concerning the possible meaning of quantization in connection with transition probabilities of a process.  A process on the phase space 
      in $\R^{2n}$ is a vector 
        $X(t , \omega)$ depending on  $t>0$ and on  $\omega$ 
        belonging to a probability space $(\Omega,\mu_\Omega)$. With every  $X(t , \omega)$, similarly to with every vector of $\R^{2n}$, we define the Weyl translation operator  $W_h (X(t , \omega))$ 
       by (\ref{trans-Weyl}). 
        We then obtain a unitary operator  $V(t , \omega)$ 
        in $L^2 (\R^n)$ depending  on $\omega$ and thus a random operator,
        $$ V(t , \omega) =  W_h (X(t , \omega)).$$
    Also, we consider an operator   $A_0$ with $F\in D_h(\R^{2n})$  as a Wick symbol. 
For any $t>0$, we define a random operator bounded in $L^2(\R^n)$ by, 
 $$ A(t , \omega) =  V(t , \omega)^{\star}  A_0  V(t , \omega).$$
 In formula (\ref{def-gene}), we define the quantization of  $F$ denoted by      
$  {\rm Op}_t^{\nu_t}(F)$.  This quantization can be viewed as the expected value of the random operator  $A(t , \omega)$. Namely,
$$  {\rm Op}_t^{\nu_t}(F) = \int _{\Omega}  A(t , \omega) d \mu_\Omega(\omega). $$
         
We note that Proposition \ref{act-proc-op} additionally  shows that there is a relation between  processes on the phase space $\R^{2n}$ an bounded operators in   $L^2(\R^n)$. 
  
 The paper is organized as follows. The results are precisely stated in Section 2, such as, definition of the general quantization,  connection with L\'evy processes, symbol composition formula and Wick quantization properties. The rest of the paper is essentially concerned with the proofs.  In Section 3, proofs of the results regarding Wick calculus  are found. In Section 4, results involved in Section 5 are stated and proven.  Section 5 is devoted to the proofs for the composition issues in the general case. Standard quantizations and L\'evy processes are considered in Section 6. In Section 7, usual notions together with some of their properties are recalled. Finally, we underline a relation between bounded measures and  classes of operators  in Section 8.
  
    {\bf Aknowledgments.} The third author is very grateful to 
D. Robert for useful discussions.

 \section{Statement of results.} \label{Statement of results}
 
We begin with some  notations. The scalar products are all denoted by $<\cdot, \cdot >$. It is linear with respect to the first variable and anti-linear with respect to the second variable. For example, the scalar product of two functions $f$ and $g$, say in  $L^2(\R^n)$, satisfies   $< f , \lambda g> = \overline \lambda <f , g>$ for $\lambda\in \C$.
     If $U = (u_1 , \dots, u_n)$ and $V= (v_1 , \dots ,v_n)$  are two  elements in $\C^n$, we set $U\cdot V = \sum_{j=1}^n u_j v_j$ and $U^2 = U \cdot U$. We denote by 
     $<X , Y>_{\rm C}$ the Hermitian product in  $\C^n$. Also, $C$ is a often a constant that may vary from line to line and independent on the running parameters under consideration.

\subsection{Wick calculus}\label{Wick calculus}

  The Wick symbol  of a bounded  operator  $A$ from ${\cal S}(\R^n)$ to ${\cal S}'(\R^n)$ is the function depending on the   parameter $h>0$,  defined on $\R^{2n}$ by,
 \be\label{symb-Wick} \sigma_h^{\rm Wick}(A) (X)  =  < A \Psi _{X h},  \Psi _{X h} >, \ee 
where the variable in $\R^{2n}$ is denoted by $X = (x , \xi)$ and with the functions $\Psi _{X h}\in L^2(\R^n)$ being the coherent states defined in (\ref{coh-state}).

 The  following notion of   confinement function  is inspired from \cite{Bo-Ler}. 
    The confinement function of a bounded operator   $A$ in 
      $L^2( \R^n)$  is a function on $\R^{2n}$ defined by,
    \be\label{confin}  d_hA (V)  =   \sup _{X-Y = V}   | < A \Psi _{X h},  \Psi _{Y h} > |,
    \ee 
for any $V\in \R^{2n}$.  
   \begin{defi}\label{Def-Ch} 
The space $C_h$ denotes the space of bounded operators $A$ in $L^2(\R^n)$ satisfying $ d_hA \in L^1(\R^{2n})$.
   \end{defi} 
Set, for any  bounded operator  $A$ and  for every $(X,Y)\in\R^{4n}$,
    \be\label{Foll}  (\sigma_h^{\rm ext}A) (X , Y) =  \frac  { < A \Psi _{X h},  \Psi _{Y h} >} { <  \Psi _{X h},  \Psi _{Y h} >}  \ee 
(see \cite{Ber75}). 
The function $\sigma_h^{\rm ext}$  is sometimes called double Berezin symbol or off-diagonal Berezin symbol. For the sake of simplicity, we call it  Berezin symbol.
    
  Note that the   denominator in the right hand side of (\ref{Foll}) is never vanishing from (\ref{PSEC}).

      \begin{theo} \label{Ch-to-D_h}  Let $h>0$. Fix   a bounded operator $A$ in $L^2(\R^n)$. Let 
      $ \sigma _h^{\rm Wick} (A)$  and  $\sigma_h^{\rm ext}A $ be respectively the Wick symbol  and 
      the  Berezin symbol of $A$ defined in  (\ref{symb-Wick}) and (\ref{Foll}). 
      Then, 
      
      i)  We have, for every $(X,Y)\in\R^{4n}$,
    \be\label{Fol-CR-et-anti}   (\partial _{x_j } + i \partial _{\xi_j }) 
    (\sigma_h^{\rm ext}A) (X , Y) = 0 ,\quad
     (\partial _{y_j } - i \partial _{\eta_j })
       (\sigma_h^{\rm ext}A)    (X , Y) = 0.\ee 
      
      ii) The following equality holds true, $$ 
       \forall\,X\in\R^{2n},\  (\sigma_h^{\rm ext}A)  (X , X) = 
    \sigma _h^{\rm Wick} (A) (X).$$
      
      iii) If  $A\in C_h$ then we have,
      $$ \forall\, (X,Y)\in\R^{4n},\ | (\sigma_h^{\rm ext}A) (X , Y)   | \leq d_hA (X-Y) 
      e^{ \frac { 1 }  { 4h } |X-Y|^2  } .$$

       \end{theo} 
       
       Theorem \ref{Ch-to-D_h} says that the Wick symbol of an operator in 
       $C_h$ belongs necessarily to the space  $D_h ( \R^{2 n })$ 
     defined below.
       
        \begin{defi}\label{Def-Dh}   For each $h>0$,  the space 
  $D_h ( \R^{2 n })$  stands for the set of   functions $F$ on $\R^{2 n }$ such that there exists an analytic  function $ (X,Y)\rightarrow F^{\rm ext} (X , Y)$ on   $\C^{2n }$ 
  having the following properties.
  
  i)  $F^{\rm ext}$ satisfies for any $(X,Y)\in\R^{4n}$,
    \be\label{CR-et-anti}   (\partial _{x_j } + i \partial _{\xi_j }) 
   F^{\rm ext} (X , Y) = 0 ,\quad
     (\partial _{y_j } - i \partial _{\eta_j })
      F^{\rm ext}    (X , Y) = 0,\ee 
  ii) We have $$F^{\rm ext}(X , X) = F(X).$$ 
  
 iii) There is 
a function $\psi_F \in L^1(\R^{2n}) \cap L^{\infty}(\R^{2n})$ verifying the estimate for all $(X,Y)\in\R^{4n}$,
       $$ | F^{\rm ext}(X , Y)| \leq e^{ \frac { 1 }  { 4h } |X-Y|^2  } 
       \psi_F (X - Y) .$$
When the conditions i) ii) and iii) are fulfilled, we set,
$$ \Vert F \Vert _{D_h( \R^{2n})}  =  \Vert \Psi_F  \Vert _{L^1 (\R^{2n})} + 
 \Vert \Psi_F \Vert _{L^{\infty } (\R^{2n})}.  $$
   
    \end{defi} 
    
 As a consequence,       
$\sigma_h^{\rm Wick}(\cdot)$ maps  $C_h$ into $D_h( \R^{2n})$. 
Before stating the converse result, let us give some properties concerning the spaces $D_h( \R^{2n})$.

    \begin{theo}\label{equiv-Dh}  
    
     i) For all $F\in D_h( \R^{2n})$,  we have, 
  
      $$ F(X) =  ( \pi h)^{-n} \int _{\R^{2n}}e^{ - \frac {1 }  {h}   |X-T|^2 }
      H_hF(T) dT, $$
         for every $X\in \R^{2n}$ where
      $$  H_hF(T) =  ( 4 \pi h)^{-n}   \int _{\R^{2n}} 
       F^{\rm ext} \Big(  T+\frac{Z}{2} , T-\frac{Z}{2}\Big ) 
      e^{ - \frac {1 }  {4h}   |Z|^2 } dZ,$$
      for any $T\in \R^{2n}$.
      
    ii) The operator $H_h$ maps $D_h (\R^{2n})$ into 
    $ L^{\infty} (\R^{2n}) $. It is  commuting with translations and its symbol equals to  $u\rightarrow e^{ \frac { h|u|^2}   {4} }$. 
   
    iii)  Let  $F\in D_h( \R^{2n})$. There exists a function $\widetilde F$ on $\C^{2n} \simeq \R^{4n} $ which is holomorphic in the following sense,     
    \be\label{Cauchy-Riem} (\partial _1 + i \partial _2 ) \widetilde F = 0,\
    (\partial _3  + i \partial _4 ) \widetilde F = 0
  \ee 
   with a restriction to $\R^{2n}$ 
     equals to $F$, that is to say,  $\widetilde F (x , 0, \xi , 0) = F(x , \xi)$. Moreover, for $X,Y\in\R^{2n}$,
      \be\label{majo-ext}   | \widetilde F (X+iY) |  \leq 
    ( 4 \pi h)^{-n}   \Vert \psi_F \Vert _{L^1 (\R^{2n})} 
  \  e^{  \frac {1}  {h}   |Y|^2 }, \ee 
that is to say, for any $(x_1 , x_2 , x_3 , x_4)\in\R^{4n}$,
  $$  | \widetilde F (x_1 , x_2 , x_3 , x_4) | \leq
  ( 4 \pi h)^{-n}   \Vert \psi_F \Vert _{L^1 (\R^{2n})} 
  \  e^{  \frac {1}  {h}  ( |x_2|^2 + |x_4|^2) }.$$ 
      
    iv) Fix $F\in D_h( \R^{2n})$. We have,  with  $X = (x , \xi)\in \R^{2n}$ and $Y = (y , \eta)\in \R^{2n}$, 
   \be\label{F-ext}   F^{\rm ext}(X , Y)=  \widetilde F \left ( 
 \frac {x+y} {2} ,   \frac {\xi - \eta } {2} , 
 \frac {\xi + \eta} {2} ,   \frac {y-x} {2}\right ) .\ee 
 If $\widetilde F$ is satisfying  (\ref{Cauchy-Riem}) then $F^{\rm ext}$ 
 is verifying (\ref{CR-et-anti}). 
 
 v)  If a function  $F$ on $\R^{2n}$  has a holomorphic extension 
  $\widetilde F$ in the sense of  (\ref{Cauchy-Riem}) on  $\C^{2n}$ which satisfies,
  \be\label{majo-ext-hol}  | \widetilde F (X+iY) |  \leq 
     | \psi_F (Y) | 
  \  e^{  \frac {1}  {h}   |Y|^2 }, \ee 
  where $\psi_F \in L^1 (\R^{2n}) \cap L^{\infty} (\R^{2n}) $, 
  then the function  $F\in D_h (\R^{2n})$. 
  
  vi) The  convolution by a bounded measure is an operator in $D_h (\R^{2n})$ having a norm smaller or equal than $1$. 
 
   \end{theo} 
   
See also \cite{Hall} for some related properties in the $L^2$ case. Throughout this paper,  we express property (\ref{CR-et-anti}) by saying that  $F^{\rm ext}$ is a sesqui-holomorphic function. 
Theorem \ref {equiv-Dh} is proven in  Section 3.2. 
One notes that if  $F(X) = e^{i X \cdot A}$ with   $A\in \R^{2n}$  then 
 $  (H_{\lambda  h}F) (X) = e^{i X \cdot A + h |A|^2} $, $X\in \R^{2n}$. 
 
 Recall that the Gevrey class of parameter $\frac{1}{2}$   on $\R^n$ is the set of all functions $f$ defined on $\R^n$ such that, there exist $K>0$ and $C>0$ satisfying for all multi-index $\alpha\in\N^n$ and all $x\in\R^n$, 
 $ | \partial _x ^{\alpha }f(x)
|\leq K 
 C^{|\alpha   |} \sqrt { \alpha ! }$.

The following result shows that the union over $h>0$ of the spaces    $D_h( \R^{2n})$ is the Gevrey class of order $\frac{1}{2}$ on $\R^{2n}$.

   \begin{theo}\label{Dh-Gev} i)  Let $F\in C^{\infty }(\R^{2n})$
   satisfying for all    multi-indices  $(\alpha , \beta)\in(\N^n)^2$,
\be\label{GevreyC} \Vert \partial _x ^{\alpha }  \partial _{\xi } ^{\beta } F 
\Vert _{L^{\infty } (\R^{2n})} \leq K 
 C^{|\alpha +\beta  |} \sqrt { \alpha ! \beta ! },\ee
  with $K>0$ and $C>0$. If $\lambda < 1/(2C^2)$ then $F\in D_{\lambda } (  \R^{2n})$. 
  
 ii)  Conversely, if  $F\in D_{\lambda } (  \R^{2n})$ (with $\lambda>0$)
then, for all multi-indices $\alpha\in\N^{2n}$,
   $$  \Vert \partial ^{\alpha} F \Vert _{L^{\infty}( \R^{2n})} 
  \leq K   \sqrt {\alpha ! }
   \Big (\frac  { C }  {\lambda } \Big )^{| \alpha| /2} 
   \Vert  F \Vert _{D_{\lambda}( \R^{2n}) }.$$
  
  \end{theo}   
      
Theorem \ref{equiv-Dh}  and Theorem \ref{Dh-Gev}   
are proven in  Section \ref{s3.2} excepted for the second inclusion in Theorem  \ref{Dh-Gev} which is proven in  Section \ref{s4.1}.

 We turn to the converse result of      
Theorem \ref{Ch-to-D_h}. 
   For any $X\in \R^{2 n }$, we use the operator  $\Pi _{X h}$ defined by, for any $f\in L^2(\R^n)$,
  $$ \Pi _{X h} f = < f , \Psi _{X h} >  \Psi _{X h}.$$

     \begin{theo}\label{bij}   Let $F\in D_h (\R^{2n})$.   
     
       i) The expression,

  \be\label{def-A-Wick}  {\rm Op}_h^{\rm Wick}(F) =
  (2 \pi h)^{-2n} \int _{\R^{4n} }  \Pi _{X h} 
   F^{\rm ext}(Y, X)  \Pi _{Y h} dX  dY\ee 
   defines a bounded operator in     $L^2 (\R^{n})$.
   
   ii) We have,  
    $$\sigma_h^{\rm Wick}( {\rm Op}_h^{\rm Wick}(F)) = F.$$
   iii)  The operator  ${\rm Op}_h^{\rm Wick}(F) $ is belonging to $C_h$ 
 and we have,
   $$ \Vert {\rm Op}_h^{\rm Wick}(F) \Vert_{C_h} 
      \leq C 
     \Vert F \Vert_{D_h ( \R^{2 n })} .$$ 
      iv) We have, 
 $$ W_h(T)^{\star}  {\rm Op}_h^{\rm Wick}(F)  W_h(T) = 
  {\rm Op}_h^{\rm Wick}(F_T),$$
  where,  for each $T\in\R^{2n}$,
  $ W_h(T) $ stands for the Weyl translation operator in $L^2 (\R^{n})$ defined in (\ref{trans-Weyl}) and
 $ F_T(X) = F(X+T)$, for all $X\in\R^{2n}$.
 
    \end{theo}

 Theorem \ref{bij} is proven in 
Section 3.3. An equivalent formulation of points $i)$ and $ii)$ of this result when using  Bargmann transform  is given by Proposition \ref{altern} in Section 7.3. Therefore,  one can say that  any bounded  operator  in
  $L^2(\R^n)$  having a confinement function belonging to $L^1( \R^{2n})$ 
is a pseudo-differential  operator. See also  \cite{U-2}.

   We now observe in Theorem \ref{comp-abs} below that, if two operators 
   $A_h$ and  $B_h$ are the Wick quantizations of 
two suitable symbols $F$ and $G$ then the composition 
  $A_h \circ B_h$ is also the Wick quantization of a symbol  $F \sharp _h^{\rm Wick}G$ which can be written with a formula based on   $F$ and  $G$.

  \begin{theo}\label{comp-abs} i) If the two  operators  $A$ and $B$ 
  belong to $C_h$ then the composition  $A \circ B$ also  lies in $C_h$ and we have for every $V\in\R^{2n}$,
     $$  d_h (A \circ B) (V) \leq   (2 \pi h) ^ {-n} \int _{\R^{2n}}  
     ( d_h A) (V-U)  ( d_h B) (U) dU. $$
     ii) Take $F$ and $G$ in $D_{ h} ( \R^{2n})$ and set
    $A_h =  {\rm Op}_h^{\rm Wick}(F)$ and  $B_h =  {\rm Op}_h^{\rm Wick}(G)$. 
    Then, there exists a function in  $D_{ h} ( \R^{2n})$ denoted by $F \sharp _h ^{\rm Wick}G$ satisfying,
    $$ \sigma  _h ^{\rm Wick} (A_h \circ B_h) = F \sharp _h ^{\rm Wick}G. $$
   In addition,  we have the identity,    for all $X\in\R^{2n}$,
     \be\label{sharp-Wick}   (F \sharp_h^{ Wick } G)  (X) = (2 \pi h)^{-n}
     \int _{ \R^ {2n}}  
   F  ^{\rm ext} ( X+T, X)   G  ^{\rm ext} ( X , X+T)   
   e^{ - \frac {1}   {2h} |T|^2 } dT.\ee 

  
   \end{theo}
  
  One however  notices that Theorem \ref{comp-abs} is not giving the composition 
 $F \sharp_h^{ Wick } G$ with a formula using  $F$ and $G$ themselves but rather using the sesqui-holomorphic extensions  $F  ^{\rm ext}$ and $G  ^{\rm ext}$ of $F$ and $G$.  In the aim to obtain a formula for the composition of $A_h \circ B_h$ written directly with $F$ and $G$ themselves without sesqui-holomorphic extensions, we have to assume that 
 $F$ and $G$ are lying is some subspace of  
   $D_{h} (\R^{2n} )$. See also \cite{Solo} for a related result.
   
   As already mentioned  in Section 1, in that purpose, we proceed with the three-steps following method: 1) We consider  the tensor product 
$ (F \otimes G) (X , Y) = F(X) G(Y)$. 2) We apply a convolution product operator (with variables in $\R^{4n}$). 3) We restrict to the diagonal  $X=Y$. 
   
  The following differential operator plays a significant role in the composition.
    \be\label{L-Wick} L^{\rm Wick} =   \frac {1}   {2} \sum _{j\leq n} 
      ( \partial _{x_j} - i \partial _{\xi_j}) 
         ( \partial_{y_j} + i \partial _{\eta_j}).  \ee 
   
 We also recall at this stage that  the space  $D_{ h} (\R^{2n} )$  is a decreasing function of  $h$.

 \begin{theo}\label{comp-Wick} Take $\lambda $ and  $\mu$  satisfying 
 $0 < \mu < \lambda $. Fix  $F$ and  $G$ belonging to
   $D_{\lambda h} (\R^{2n} )$. 
   Then, 
    
    i) If  $\lambda - \mu  $ is large enough then the function,
     
 $$  e^{h L^{\rm Wick} } (F \otimes G)= \sum _{m\geq 0} 
  \frac {h^m}   {m!} ( L^{\rm Wick})^m (F \otimes G) 
  $$
 is well defined on  $\R^{4n}$
 as an absolutely convergent series. Moreover, it is defining a function in 
 $D_{\mu h} (\R^{4n} )$.  
 
 ii) Set $\mu = 1$.  Let $R$ be the restriction to the diagonal operator,  
 namely, $(RF) (X) = F(X , X)$. Then, for any   $F$ and  $G$ in $D_{\lambda h} (\R^{2n} )$ with $\lambda $ large enough,   we have,
 \be\label{comp-Wick-3}  F \sharp_h^{ Wick } G = R  e^{h L^{\rm Wick} } (F \otimes G).\ee

   \end{theo} 
   
  Theorem \ref{comp-Wick} is proven in Section \ref{s3.4}. Its statement is generalized  in Section \ref{s2.3}.

This result can be formulated  differently.  This is the standard Mizrahi series \cite{Mizrahi} but for symbols in $D_{\lambda} (\R^{2n} )$. See also \cite{Appleby}.

   \begin{theo}\label{Miz}  Take  $F$ and $G$ belonging to 
   $D_{\lambda h} (\R^{2n} )$ with $\lambda >0$. If
   $\lambda $ is sufficiently large then the series in the right hand side of (\ref{comp-Miz}) is absolutely convergent and we have,    
\be\label{comp-Miz}( F \sharp _h^{\rm Wick} G ) (X) =  \sum _{\alpha} 
    \frac  {h^{|\alpha|}}   { 2^{|\alpha|} \alpha! }  
        \partial _X^{\alpha}   F  ( X)  \ 
         \partial _{\overline X}^{\alpha}  G (X), \ee
where
     $$ \partial _{X_j} = \partial _{x_j} - i \partial _{\xi_j},\quad
      \partial _{\overline X_j} = \partial _{x_j} + i \partial _{\xi_j} .$$
  
   \end{theo} 
   
Theorem \ref{Miz} is proven in  Section \ref{s3.4}.

 \subsection{Other pseudo-differential calculi.} \label{2.2} 
 
 The 
 Weyl, anti-Wick, classical and   Born Jordan calculi are recalled in Section \ref{s6}.
For all these quantizations,  one assigns an operator  ${\rm Op}_h (F)$ to a symbol $F$ chosen in a suitable class. 
Propositions \ref{mes-Weyl},  \ref{mes-AW},  \ref{mes-class} 
and \ref{BJandW}  show that,
 \be\label{prob} < {\rm Op}_h (F) \Psi _{X h},  \Psi _{X h} > = (F \star \nu) (X) ,\ee 
with measures given below for each of these quantizations,

    \begin{align}
 &   d \nu_h^{\rm Weyl} (X) =  ( \pi h) ^{-n}
  e^{ - \frac {1} {h} |X|^2 } dX 
  \label{mesure-Weyl} \\[3mm]
&   d \nu_h^{\rm AW} (X) =  (2 \pi  h) ^{-n}
  e^{ - \frac {1} {2h} |X|^2 } dX  
    \label{mesure-AW}\\[3mm]
&    d \nu_h^{\rm class} (x , \xi) =  2 ^{-n/2}  ( \pi h) ^{-n}
     e^{ - \frac {1} {2h} (|x|^2 + |\xi|^2 + 2i x \cdot \xi)}  dx d\xi 
     \label{mesure-class} \\[3mm]
& d \nu_h^{\rm BJ} (x , \xi) = 
   \Phi (x , \xi) dx d\xi \label{mesure-BJ}
   \end{align}
   where 
   $$ \Phi (x , \xi) = (  \pi h)^{ -n }
       \int _{-1/2} ^{1/2}
      (1 + 4 \sigma^2)^{ -n /2} 
       \exp \big (  - \frac { 1}   { h (1 + 4 \sigma^2) } 
       ( |x|^2 +    |\xi |^2  + 4i \sigma x \cdot \xi )  \big )  
       d \sigma $$
    with  $X = (x , \xi)\in\R^{2n}$.

In the aim to get a  generalization of  this fact, we suppose that a bounded measure $\nu$   on   $\R^{2n}$ and a function  $F$ on  $\R^{2n}$ are given. Then, we look for an operator  $A$  satisfying  (\ref{prob}) for all 
 $X\in \R^{2n}$.

  \begin{theo} Let   a function $F$ on   $\R^{2n}$ and  a bounded measure $\nu$ on  
    $\R^{2n}$  be such that   $F \star \nu$ belongs to 
  the space   $D_h (\R^{2n} )$  (see Definition \ref{Def-Dh}) and therefore has  a sesqui-holomorphic extension  $(F \star \nu)^{\rm ext}(X , Y) $.  Then, the 
  following equality is defining a bounded operator,
   \be\label{def-gene}  {\rm Op}_h^{\nu}(F) =  (2 \pi h)^{-2n}  \int _{\R^{4n} }  \Pi _{X h} 
   (F \star \nu)^{\rm ext}(Y, X)  \Pi _{Y h} dX  dY,\ee 
   where $ \Pi _{X h} $ is the orthogonal projection on the space spanned by the coherent space $ \Psi _{X h}$.  Moreover, we have for all $X\in\R^{2n}$,
 $$ < {\rm Op}_h ^{\nu} (F) \Psi _{X h},  \Psi _{X h} > =
  (F \star \nu) (X).$$ 
  \end{theo}

   This is the common quantization formula shared by 
 the Weyl, anti-Wick, classical and Born-Jordan calculi, 
   and also by all other calculi defined  starting from  L\'evy processes.
   
This formula allows us to define an operator ${\rm Op}_h ^{\nu} (F) $ 
for any $F\in D_h (\R^{2n} )$ and every bounded measure  $\nu$ on $\R^{2n} $ without any additional hypotheses. This is a consequence of   of Theorem  \ref{equiv-Dh} $vi)$.

For some suitable measures $\nu$, the convolution  product  $ F \star \nu $
 can belong to  $D_h (\R^{2n} )$ without assuming that   $F\in D_h (\R^{2n} )$. Then, we can define quantization operators associated with symbols $F$ that are in larger class than $D_h (\R^{2n} )$, such as  the $L^{\infty }(\R^ {2 n })$ class  for example. 
    
 We now examine the case where we can quantize symbols
    in $L^{\infty } (\R^{2n} )$. Theorem  \ref{holo-convol} below,
    concerning the class of measures of Definition \ref{def-op-D}, can be applied in different cases such that the case where the measure is a transition probability of a L\'evy process with a definite positive diffusion matrix (see Section \ref{s2.4}). Regarding the anti-Wick quantization, it is known that one can use 
symbols in $L^{\infty } (\R^{2n} )$ but we shall show that this calculus can also be defined with formula   (\ref{def-gene}). 
 
  \begin{defi}\label{def-op-D}   Set $\gamma < 0$. We denote by  
  ${\cal M} ( \gamma , \R^{2n} )$ the set of families ($\nu_h)$ of bounded   measures on   $\R^{2n}$ satisfying,
    \be\label{Four-mes}
     \widehat  \nu_h (u) = e^{h\varphi (u)}, \ee  
where   $\varphi$  is a function on   $\R^{2n}$ 
    written as $\varphi = \varphi_1 + \varphi_2  $ such that,

   i) the function $\varphi_1 $ is the symbol of a bounded operator in  $L^{\infty } (\R^{2n} )$ which is commuting with  translations.

   ii) the function $\varphi _2 \in C^{\infty } (\R^{2n} )$.  For every $\alpha\in \N^{2n}$, there is   $A_{\alpha }>0$ such that,  for any $u\in\R^{2n}$,
   \be\label{majo-der-phi} | \partial ^{\alpha }\varphi_2 (u) |
    \leq A_{\alpha } (|u|^2 +1).
  \ee
For every $u$ large enough,
      \be\label{majo-phi}  {\rm Re}\, \varphi_2(u) \leq  \gamma |u|^2. \ee


      \end{defi} 
      
  If $T$ is a  translations commuting bounded  operator in 
 $L^{\infty } (\R^{2n} )$ 
 and  if $F(X) = e^{i a \cdot X}$ then $TF$ is written under the form
   $TF (X) = \varphi (a) e^{i a \cdot X}$ and $\varphi$ is called symbol of $T$. We then often write $T = \varphi (D)$. For instance, it is sufficient that
    $\varphi_1$ above is the Fourier transform of a 
   $L^{1 } (\R^{2n} )$ function. 
    
 The set of measures ${\cal M} ( \gamma , \R^{2n} )$ in Definition   \ref{def-op-D} is containing transition probabilities of L\'evy process with  define positive diffusion matrices (see Theorem \ref{grano-salis}).

  \begin{theo}\label{holo-convol} Let $(\nu_h )\in {\cal M} ( \gamma , \R^{2n} )$  with $\gamma <0$.  Suppose that $  \delta < 4 |\gamma |$. Then, the following properties hold true.

  i) The measure  $\nu_h $ is absolutely continuous with respect to the Lebesgue measure. Its density $\Phi_h$ has a sesqui-holomorphic extension satisfying for any $Y\in\R^{2n}$,  
  $$ \int _{\R^{2n}}  | \Phi_h (X+ i Y) | dX \leq C e^{  \frac {1} {\delta h} |Y|^2 } .$$
  
  ii) The convolution product by  $\nu_h$
     maps $L^{\infty} ( \R^{2n})$ into $D_{\delta h} ( \R^{2n})$. 
      
  iii) The symbol of this convolution operator by  $\nu_h$ is equal to the function
   $  e^{ h \varphi (u)} $. 
     
      \end{theo}   
      
The convolution operator by   $\nu_h$ is denoted 
    $  e^{ h \varphi (D)} $ in the sequel  by abuse of notations.

  As a consequence,  if $4 |\gamma | >1$ and if $F\in L^{\infty} (\R^  {2n} )$ then the convolution product $F \star \nu_h\in D_{ h} ( \R^{2n})$. This point is directly  implying that the quantization operator  
   $ {\rm Op}_h^{\nu}(F)$ can be defined by formula  (\ref{def-gene}).

 This can be applied to some (but not all) usual pseudo-differential calculi. One knows that, for any symmetric matrix   with positive definite real part, 
 $$  \int _{\R^{2n}} e^{- (AX)\cdot  X - i  X \cdot u  } dX = C 
  e^{-\frac  {1}  {4}   (A^{-1}u ) \cdot u }, $$
  for all $u\in   \R^{2n}$.
In particular,  the measures  $\nu_h^{\rm Weyl}$,  $\nu_h^{\rm AW}$ 
 and   $\nu_h^{ class }$  defined in (\ref{mesure-Weyl}),
  (\ref{mesure-AW} )  and  (\ref{mesure-class}) are all written 
under the form  (\ref{Four-mes}) with,
 $$ \varphi ^{\rm AW} (u) = - \frac  {1}  {2}  |u| ^2,\quad 
  \varphi ^{\rm Weyl} (u) = - \frac  {1}  {4} |u| ^2,\quad  \varphi ^{\rm class} (x, \xi) =  - \frac  {1}  {4} 
   ( |x| ^2 + |\xi| ^2 - 2i x \cdot \xi ). $$ 
We therefore observe that Theorem \ref{holo-convol} can be applied for the anti-Wick calculus. Theorem \ref{holo-convol} cannot be used for the Weyl quantization since the condition  $4 |\gamma |>1$ is not fulfilled. We note here that this later point is agreeing with the well known fact  that  a Weyl quantized operator cannot be associated with a symbol which is only $L^{\infty} (\R^  {2n} )$. Thus, Theorem \ref{holo-convol} cannot be applied  with the Weyl and also with the Born-Jordan calculi since it is known that $L^{\infty}( \R^ {2n})$ is a class of symbol that is too large, but nevertheless, we can still always quantize symbols in $D_h(\R^{2n})$. We remark at this stage, in the purpose to get wider classes of symbols, that one should consider  specific  classes adapted to each calculus, that is, classes of symbols depending on the mesures $(\nu_h)$ parametrizing the quantization. Even if the result concerning the Weyl quantization is standard, we give a proof that may have its own interest.

  \begin{theo}\label{S-Weyl}    
   Let   $S_W$ be the space of  tempered distributions  
  $F$ such that the function,
  $$ ( {\cal V }_W F) (X, Y ) = (
  \pi  h) ^{-n} \int _{\R^{2n}}  F  ( T  ) 
   e^{ - \frac {1} {h} |X- T|^2 - \frac {2i} {h} Y \cdot (X-T) } dT $$
is continuous and satisfies,
   $$ |   ( {\cal V }_W F)  (X , Y) | \leq \psi (Y),$$
 with the function $\psi\in L^1(\R^{2n}) \cap L^{\infty}(\R^{2n})$. 
   Then, for all $F\in S_W$, $ F \star \nu_h ^{\rm Weyl} $ 
   is belonging to  $D_{h} ( \R^{2 n })$ 
 and for some $C>0$,
    $$ \Vert F \star \nu_h ^{\rm Weyl} \Vert _{D_{h} ( \R^{2 n })} \leq C 
   ( \Vert \psi  \Vert _{L^1(\R^{2n})}  
   + \Vert \psi  \Vert _{ L^{\infty}(\R^{2n})} ).   $$
Consequently,  the operator 
${\rm Op}   _h ^{\rm Weyl} (F)  $ can be defined by 
(\ref{def-gene}) for every  $F\in S_W$.

   \end{theo} 
   
   In  \cite{GRO}  (see Proposition \ref{S-Weyl}), the class of symbols $S_W$ is a modulation space denoted by $M_g ^{\infty, 1} $. The class $S_W$ 
    corresponds to   $d = 2n$, $g(T) =  e^{ - \frac {1} {h} | T|^2 }$ 
    and $m(T) = 1$. One notices that this class of Gr\"ochening is closely related to the class of Sj\"ostrand in   \cite{SJ-1}.  Also note that every function   $F\in C^{\infty} (\R^{2n} )$ being bounded together with all of its derivatives
  is belonging to  $S_W$.

   {\it Proof of Theorem \ref{S-Weyl}.} 
   
    We have,
   $$  (F \star \nu_h ^{\rm Weyl}) (X) = ( \pi h) ^{-n} 
    \int _ {\R^{2n}}  e^{ - \frac {1} {h} |X-T|^2 } 
    F(T) dT. $$
  We observe that  this function has a sesqui-holomorphic extension given by, for all $(X,Y)\in\R^{4n}$,
 \begin{align}\nonumber
  \widetilde { (F \star \nu_h ^{\rm Weyl})} (X+iY)& = 
    ( \pi h) ^{-n} 
    \int _ {\R^{2n}}  e^{ - \frac {1} {h} (X+ i Y-T)^2 } 
    F(T) dT
    \\[3mm] &=  e^{  \frac {1} {h} |Y|^2 } 
    (WF) (X , Y). \end{align}
Consequently, we get, 
    $$  | \widetilde {( F \star \nu_h ^{\rm Weyl})} (X+iY) | \leq 
    e^{  \frac {1} {h} |Y|^2 } | \psi (Y) |. $$
   Therefore,  $F \star \nu_h ^{\rm Weyl}$ is indeed in  $D_{h} ( \R^{2 n })$ according to point v) of Theorem  \ref{equiv-Dh}.

   \fpr 
   
  Concerning Born-Jordan calculus, we can quantize a symbol in $D_h (\R^{2n})$ to get a bounded operator in 
 $L^2(\R^n)$ but we cannot consider larger classes of symbols. Let us mention that  Born-Jordan  bounded operators from ${\cal S} (\R^n)$ into 
  ${\cal S}' (\R^n)$ are studied   in 
  \cite{d-Gos} 
  and  \cite{d-Gos-Luef, d-Gos-Luef-2}.

   \subsection{Composition of symbols. }\label{s2.3}
   
  Let   
   $A_h =  {\rm Op}_h^{\nu_h}(F)$ and  $B_h =  {\rm Op}_h^{\nu_h}(G)$ 
  be two operators defined by the general quantization   (\ref{def-gene}), with two symbols $F$ are $G$ 
in  $D_{\lambda h} (\R^{2n} )$ ($\lambda $ sufficiently large) and with a family of measures  $(\nu_h)\in {\cal M} ( \gamma , \R^{2n} )$  (where  $\gamma <0$). Our goal is to prove that   $A_h \circ B_h = 
   {\rm Op}_h^{\nu_h}(F \sharp^{\nu} G)$, where  the symbol $F \sharp^{\nu} G$ 
   is  given by an explicit formula.   
   In that aim, we follow the second method in Section \ref{Wick calculus}, that is: tensor products, action of an operator with variables in 
 $\R^{4n}$ 
commuting with translations, 
  restriction to the diagonal. 
   
 Thus, we consider   
$(\nu_h)\in {\cal M} ( \gamma , \R^{2n} )$ with $\gamma <0$. We have (\ref{Four-mes})
 where  $\varphi$ satisfies the hypotheses 
 (\ref{def-op-D}). We then define $ L ^{\nu} $ on  $\R^{4n} $ 
  by,
  \be\label{Phi-nu}  L ^{\nu}  (x , \xi , y, \eta) =
   \varphi (x , \xi) + \varphi (y , \eta)
    - \varphi (x+y , \xi + \eta) + L^{\rm Wick} (x , \xi , y, \eta) ,\ee 
  where
   \be\label{symb-LWick}   L^{\rm Wick} (x , \xi , y, \eta) = - \frac {1}  {2}
    \sum _{j=1}^n ( x_j - i \xi_j) (y_j + i \eta _j) .\ee 
 We notice that there exists   $m\in  \R$ such that the function  $ L ^{\nu}$  satisfies the hypotheses of Theorem  \ref{expo} and in particular,  
 for any $u\in\R^{4n}$,
 $$  {\rm Re} L^{\nu} (u) \leq  m |u|^2. $$

  \begin{theo} \label{expo} Let  $\Phi$ be a function on $\R^{4n}$ 
  written as  $\Phi = \Phi _1  +  \Phi _2 $.  
 Suppose that  the function $\Phi_1 $ is the symbol of a bounded operator in  $L^{\infty } (\R^{4n} )$  commuting with  translations. 
Assume that the function $\Phi_2 \in C^{\infty } (\R^{4n} )$ and that, for any $\alpha\in \N^{4n}$,
   there exists a real number $A_{\alpha }$ such that, for every  $u\in\R^{4n}$,
   \be\label{majo-der-Phi} | \partial ^{\alpha }\Phi_2 (u) | \leq A_{\alpha } (|u|^2 +1).
  \ee
Also suppose that there exists $m  \in \R$  such that,  for any $u\in\R^{4n}$ large enough,
      \be\label{majo-Phi}  {\rm Re} \,\Phi_2 (u) \leq  m |u|^2 .\ee

Set  $\lambda >0 $ and  $\mu >0$ satisfying,
   \be\label{m-lambda-mu}   m < \lambda,\quad 
   0 < \mu <  4 ( \lambda -m) . \ee 
Then, there exists an operator  $T$ bounded from $ D_{\lambda h}(\R^{4n})$ 
  into $ D_{\mu h}(\R^{4n})$ which is commuting with  translations and with  symbol equals to
  $e^{h \Phi (u) }$.  
   
 \end{theo} 
 
By abuse of notations, the operator $T$ in  
Theorem \ref{expo} is denoted   $ e^{h \Phi  (D) }$ in the sequel.

Recall that,  for each function 
  $\Phi$ on $\R^{4n} $,  $R \Phi (X) = \Phi (X , X)$, 
  $X\in \R^{2n}$ and that,
  we use also the standard notation  $(F \otimes G) (X_1, X_2)= F(X_1) G(X_2)$, for all functions $F$ and $G$ on $\R^{2n} $. 
   
 We are  ready to give the statement of the main result concerning compositions of symbols.   

   \begin{theo}\label{compo-gene} Fix    $(\nu_h)\in {\cal M} ( \gamma , \R^{2n} )$. Set  $\varphi$ a function satisfying  (\ref{Four-mes}). Let  $L ^{\nu}$ be the function defined by (\ref{Phi-nu}).  Take $m >0$ verifying (\ref{majo-Phi})
  and choose $\lambda > m + \frac{1}{2}$.
 Let  $F$ and  $G$ belong to  $ D_{\lambda h}(\R^{2n})$. Then,

   i) The function $F \sharp _h^{\nu} G$ defined by,
   $$ F \sharp _h^{\nu} G = R   e^{h L ^{\nu} (D) } (F \otimes G) $$
   is well defined in   $ D_{ h}(\R^{2n})$.

   ii) Let  $A_h =  {\rm Op}_h^{\nu_h}(F)$ and 
   $B_h =  {\rm Op}_h^{\nu_h}(G)$. We have,
   $$ A_h \circ B_h =   {\rm Op}_h^{\nu_h} ( F \sharp _h^{\nu} G). $$

    \end{theo} 
    
   In addition, we have the expected  following  properties.
    
   \begin{theo}\label{quantif} The two below properties are valid.
    
    i) If $F = 1$ then 
    $ {\rm Op}_h^{\nu}(F) = I$. 
    
    ii) We have, 
    $$ F \sharp_h^{\nu}  G - G \sharp_h^{\nu} F= 
     \frac {h}  {i} \{ F , G \}  + {\cal O} (h^2) ,$$
    where $\{ F , G \}$ stands for the Poisson bracket,
     $$ \{ F , G \} (x , \xi) = \partial _{\xi }F \cdot \partial _{x } G -  
      \partial _{x }F \cdot \partial _{\xi  } G .$$

    \end{theo}

 If the function $\varphi$ associated with  the family of measures is a  quadratic form then we have the following result.
   
    \begin{theo}\label{comp-symb-quad}  Let $(\nu_h)\in {\cal M} ( \gamma , \R^{2n} )$  with $\gamma <0$. Suppose that  the function 
  $\varphi $ in  (\ref{Four-mes}) is a quadratic form, namely,
   \be\label{phi-D} \varphi (u) = - \frac  {1}  {2} 
     \sum _{jk} a_{jk} u_j u_k = - \frac  {1}  {2}  (Au ) \cdot  u , \ee 
  where   $A=(a_{ij})$ is a symmetric  matrix with a   definite positive 
  real part.   Then,

  i) We have for every  $F$ and  $G$ in  $ D_{\lambda h}(\R^{2n})$ 
  (with large enough $\lambda $),
  $$ F \sharp _h ^{\nu_h} G = R e^{h L^{\nu}(D)} (F \otimes G) ,$$
 where  $L^{\nu}$ is a quadratic form whose matrix also denoted by 
  $L^{\nu}$ is written as,
  \be\label{Lnu} 
    L^{\nu} =  \frac  {1}  {2} 
     \begin{pmatrix} 0 & B +  A   \\  B^T +  A  &  0  \end{pmatrix} 
    ,\quad B  = - \frac  {1}  {2} 
       \begin{pmatrix} 1 & i \\ -i &   1 \end{pmatrix}.  \ee

       ii) For any  $F\in  D_{\lambda h}(\R^{2n})$
  (with $\lambda $ sufficiently large), we have the estimates, 
   $$ \frac {h^m}  {m!} \Vert ( L^{\nu}(D)) ^m F \Vert_{ D_{  h }( \R^{2n})} 
  \leq  \frac {C^{m+1}}  {(\lambda - 1) ^m}
  \Vert  F \Vert_{ D_{\lambda  h }( \R^{2n})}. $$
    
 \end{theo}

  {\it Proof  of Theorem \ref{comp-symb-quad}. }  Equality (\ref{Lnu}) is a direct consequence of (\ref{Phi-nu}) if $\varphi$ is defined by 
  (\ref{phi-D}) and $L^{\rm Wick}$ by (\ref{symb-LWick}). 
Point ii)  comes from Theorem \ref{puiss-ordre2}.

\fpr

       \subsection{Example:  L\'evy processes.  } \label{s2.4}

 One can see, {\it e.g.}, \cite{Sato} for L\'evy processes.

We use in the sequel the   L\'evy Khintchine representation for the characteristic functions of the distributions of  L\'evy processes. More precisely, for each   
 L\'evy process  $(X_t)$   on $\R^{n}$, one has for  
all $u \in \R^n$, 
 \be\label{def-LK} \E [ e^{i < X_t , u>} ] =  e^{t \varphi (u)},  \ee 
with
 \be\label{def-LK-1} \varphi (u) =  i m \cdot u  -   \frac {1}  {2} (au )  u> + 
    \int _{ \R^{2n} \setminus \{ 0 \} } \Big ( e^{i u \cdot y } -1 - i \frac   {u \cdot y } {1 + |y|^2} \Big ) 
     d \rho (y),  \ee
 where $m\in \R^{n}$, $a$ is a semi definite positive real symmetric matrix and  $\rho$ is a measure on   $\R^{n} \setminus \{ 0 \}$  satisfying,
     \be\label{LK2} \int _{ \R^{2n}}  \inf (1, |y|^2) d\rho (y) < + \infty. \ee
  
   The matrix $a$ is called diffusion matrix, the vector  $m$ is the linear drift and the measure  
 $\rho$ is the  L\'evy measure.
    See, for instance,   \cite{Ap-Co}.

 Let $(\nu_t)$ be the measure  (transition probability) on  $\R^{n}$ defined for any measurable set $E$ in 
    $\R^{n}$ and  for each $t>0$  by,
     \be\label{mes-proc}  \nu_t (E) =  \mathbb{P} [ X(t) \in E  ].\ee 
    One has,
      \be\label{Four-mes-proc} 
     \widehat  \nu_t (u) = e^{t\varphi (u)}. \ee 

       \begin{theo}\label{grano-salis} Set  $(X_t)$  a L\'evy process on $\R^{n}$ satisfying (\ref{def-LK}) and (\ref{def-LK-1}). Suppose that the diffusion matrix satisfies,
    \be\label{FQ-def-pos}   \frac {1}  {2} (au)  \cdot  u \geq \delta |u|^2,   \ee
with $\delta >0$.  Let   $\nu_t$ be the measure  satisfying 
    (\ref{Four-mes-proc}). 
      Then, replacing  $h$ by $t$, the family of measures  $(\nu_t)$ belongs to the set
    ${\cal M} ( \gamma , \R^{n} )$  if $0 < \gamma <  \delta $. 
As a consequence,  if $\varepsilon < 4 \gamma$,  then the measure 
    $\nu_t$ has a density $\Phi_t$ with respect to   the Lebesgue measure  (transition probability density). 
Moreover,  the function   $\Phi_t$ has a holomorphic extension verifying,
    \be\label{densit-trans}  \int _{\R^{n}}  | \Phi_t (X+ i Y) | dX \leq C e^{  \frac {1}
     {\varepsilon  h} |Y|^2 }. \ee

      \end{theo}   
      
      We now restrict ourselves to L\'evy processes on the   phase space $\R^ {2n}$. From Theorem \ref{holo-convol}, we have the following result.
    
     \begin{theo}\label{quantif-Levy} Set  $(X_t)$  a L\'evy process 
     on $\R^{2n}$ satisfying (\ref{def-LK}) and (\ref{def-LK-1}) whose diffusion matrix  $a$ verifies     
    \be\label{FQ-def-pos-2}   \frac {1}  {2}(au)  \cdot  u \geq \delta |u|^2 ,  \ee
 with $\delta >0$.  Let   $\nu_t$ be the measure  satisfying 
    (\ref{Four-mes-proc}).  If $\varepsilon < 4 \delta$ then the 
    convolution operator  by the measure  $\nu_t$ maps  $L^{\infty} ( \R^{2n})$ into
     $D_{\varepsilon  h} ( \R^{2n})$. 
       
         \end{theo}

         Consequently, if $4 \delta >1$ then we can define the general quantization operator  ${\rm Op}_h^{\nu_t}(F)$ using formula (\ref{def-gene}) for any  $F\in L^{\infty} ( \R^{2n})$ (setting  $h=t$).

      \section{Wick calculus.}  \label{s3}
   
  This section  is devoted to the proofs of the results stated in Section 2.1.

   \subsection{Berezin symbol. Proof of Therorem \ref{Ch-to-D_h}. } \label{s3.1}

 {\it Proof of Therorem \ref{Ch-to-D_h}.}   
    
    Points $i)$ and $ii)$ are already  proven in \cite{Fol} (Proposition 1.68 or in  (1.69)).  
  We observe that point i) also directly comes from the fact that, according to (\ref{PSEC}), we have for every  $(X,Y)\in\R^{4n}$, with $X = (x , \xi)$ and 
    $Y = (y , \eta)$,
     $$ \frac { \Psi_{X h} \otimes \overline {\Psi_{Y h} } }
        { < \Psi_{X h}  , \Psi_{Y h}  >  } (u , v)=\hskip 12cm   $$ 
        $$ e^{ -  \frac {1}  {2h} ( |u|^2 + |v|^2 )  +  \frac {i}  {h}  
  ( u \cdot (x + i \xi)  -   \frac {i}  {h}  ( v \cdot ( y -  i \eta) }    
      e^{ -  \frac {1}  {4h} (x+ i \xi) \cdot  (x+ i \xi) 
    -  \frac {1}  {4h} (y  -  i \eta) \cdot  (y - i \eta ) 
     -  \frac {1}  {2h}  (  x +  i \xi) \cdot  (y - i \eta ) } .$$

$iii)$   Set $A\in C_h$ and let $\sigma_h^{\rm ext}A$ be its Berezin symbol defined in 
     (\ref{Foll}).   According to (\ref{PSEC}), we have for any $(X,Y)\in\R^{4n}$,
 \begin{align}\nonumber
  | \sigma_h^{\rm ext}A (X , Y)| &= | < A \Psi_{X h} , \Psi_{Y h}> | \ 
      e^{  \frac {1}  {4h} |X - Y|^2 } \\[3mm]
     & \leq d_hA (X-Y)  e^{  \frac {1}  {4h} |X - Y|^2 }. 
      \end{align}
      \fpr
  
   \subsection{Proofs of Theorem  \ref{equiv-Dh} and Theorem \ref{Dh-Gev} $i)$.  } \label{s3.2}

     {\it Proof of Theorem  \ref{equiv-Dh}.} 
     
  $i)$   Take   $F\in D_h( \R^{2n})$. 
     By (\ref{CR-et-anti}),  the function  $ F^{\rm ext}$ is harmonic and therefore  invariant by the heat kernel. Therefore, for each $\theta >0$
  $$  F^{\rm ext}(X , Y) = \frac {\theta ^{2n}}  { \pi ^{2n}} 
   \int _ { \R^{4n}} e^{- \theta ( |X-U|^ 2 + |Y-V|^2) } 
   F^{\rm ext}(U , V) dU dV .$$
In particular, for $Y=X$,
 \begin{align}\nonumber
   F(X ) &= \frac {\theta ^{2n}}  { \pi ^{2n}} 
   \int _ { \R^{4n}} e^{- \theta ( |X-U|^ 2 + |X-V|^2) } 
   F^{\rm ext}(U , V) dU dV\\[3mm]
   & = \frac {\theta ^{2n}}  { \pi ^{2n}} 
   \int _ { \R^{4n}} e^{- 2 \theta ( |X-(U+V)/2|^ 2 
     -  \frac {\theta }  { 2}  |U-V|^2) } 
   F^{\rm ext}(U , V) dU dV .
         \end{align}
   Choose $\theta = 1/(2h) $. We obtain,
 \begin{align}\nonumber
   F(X ) &=   (2 \pi h)  ^{-2n}  
    \int _ { \R^{4n}} 
     e^{- \frac { 1 }  { h} ( |X-(U+V)/2|^ 2 
      -  \frac { 1 }  { 4h}  |U-V|^2) } 
       F^{\rm ext}(U , V) dU dV\\[3mm]
       & =  (\pi h)  ^{-n}  \int _{\R^{2n}}e^{ - \frac {1 }  {h}   |X-T|^2 }
      H_hF(T) dT .
      \end{align}

$ii)$   We notice that,
  $$  |H_hF(T)| \leq  ( 4 \pi h)  ^{-n}  \int _{\R^{2n}} \psi_F(S) dS. $$
  For the symbol, we also observe that if  $F(X) = e^{i u \cdot X}  $
then $F^{\rm ext} ( T + Z/2, T- Z/2)  = e^{i u \cdot X 
  + \frac {1 }  {2} \sigma ( u , Z) }$ where $\sigma $ is the symplectic form.

$iii)$    We deduce from  $i)$  that  $F$ has a holomorphic extension defined by,
      $$ \widetilde F(X+iY) =  (\pi h)  ^{-n}  \int _{\R^{2n}}
      e^{ - \frac {1 }  {h}   (X+ i Y-T)^2 }
      H_hF(T) dT. $$
Thus, 
      $$ | \widetilde F(X+iY) | \leq  (\pi h)  ^{- n}   e^{ \frac {1 }  {h}  |Y| ^2 } 
       \int _{\R^{2n}}   e^{ - \frac {1 }  {h}   (X-T)^2 } 
         |  H_hF(T)   | dT.  $$
         
$iv)$ The proof of is straightforward.

      $v)$ If a function  $F$ on $\R^{2n}$ has a holomorphic extension $\widetilde F$ sur $\C^{2n}$ satisfying 
         (\ref{majo-ext-hol}) then the function $ F^{\rm ext}$ 
         defined in  (\ref{F-ext}) is sesqui-holomorphic and its restriction to the diagonal is equal to  $F$. Also, $ F^{\rm ext}$  verifies,
       $$ | F^{\rm ext}(X , Y)| \leq e^{ \frac { 1 }  { 4h } |X-Y|^2  } 
       \psi_F (X - Y) .$$
Therefore, the function $F$ belongs to $D_h(\R^{2n})$. 
       
$vi)$   If  $F\in D_{h} ( \R^{2 n })$  has a sesqui-holomorphic extension  $F^{\rm ext}$ then  $F \star \nu$ 
 also  has sesqui-holomorphic extension given by,
   $$  (F \star \nu) (X , Y) = \int _{\R^{2 n }} F^{\rm ext} (X - T , Y-T) 
   d\nu (T).$$       
   \fpr

     {\it Proof of Theorem \ref{Dh-Gev} i).}
     
  Let $F$ be a $C^{\infty }$ function on $\R^{2n}$  satisfying for all    multi-indices  $(\alpha , \beta)$, the inequality
$$ \Vert \partial _x ^{\alpha }  \partial _{\xi } ^{\beta } F 
\Vert _{L^{\infty } (\R^{2n})} \leq K 
 C^{|\alpha +\beta  |} \sqrt { \alpha ! \beta ! }.$$
    We define a function $\widetilde F $  on $\C^{2n}$ by,
  $$ \widetilde F (X+iY) = \sum \frac 
   {\partial _x ^{\alpha}  \partial _{\xi } ^{\beta}   F(X)} 
   {\alpha ! \beta !} i^{|\alpha + \beta |} 
   y^{\alpha } \eta ^{\beta} ,$$
   for any $X=(x,\xi)\in\R^{2n}$ and $Y=(y,\eta)\in\R^{2n}$.
    Thus,
    $$ | \widetilde F (X+iY) | \leq K \sum 
    \frac { | y^{\alpha } \eta ^{\beta}|}  {\sqrt {\alpha ! \beta ! } }
     C^{|\alpha +\beta  |}.
      $$
 For every $x>0$, one has, 
     $$\sum _{m \geq 0} 
     \frac {x^m} { \sqrt {m!}}  \leq
     \left ( \sum _{m \geq 0} 
     \frac {2^m x^{2m}} {m!}  \right )^{1/2} 
     \left (\sum _{m \geq 0}  2^{-m} \right )^{1/2} \leq 2^{1/2} e ^{x^2}.$$
 Therefore,
  \be\label{majo-ext-2}   | \widetilde F (X+iY) |  \leq  K 2^n 
 \  e^{   C^2  |Y|^2 } . \ee 
 \fpr

    \subsection{Operators. Proof of Theorem \ref{bij}. } \label{s3.3}

  Proof of Theorem \ref{bij}.
      
$i)$ We prove that equality  (\ref{def-A-Wick}) 
     indeed defines a bounded  operator  $A$ in $L^{2}(\R^{n})$. 
    For every  
    $f$ and  $g$ in  $ { \cal S} (\R^n)$, we have,
    $$  < Af , g>  =(2 \pi h)^{- 2n}  \int _{\R^{4n} } < f , \Psi _{X h} > \ 
    <  \Psi _{X h} , \Psi _{Y h} > \  <  \Psi _{Y h} , g >
     F^{\rm ext}(Y , X)  dX dY. $$
 From (\ref{PSEC}), 
     $$ |  <  \Psi _{X h} , \Psi _{Y h} > | \leq 
     e^{ - \frac {1}   {4h} |X - Y|^2  } $$
     and since
     $F\in D_h(\R^{2n} )$,
       $$ | F^{\rm ext}(Y , X)| \leq e^{ \frac { 1 }  { 4h } |X-Y|^2  } 
       \psi_F (X - Y) $$
    where $\psi_F\in L^1(\R^{2n}) \cap L^{\infty}(\R^{2n})$. 
 Consequently,
     $$ | < Af , g>  | \leq (2 \pi h)^{- 2n} \int _{\R^{4n} } 
    \big | \psi_F (X - Y) < f , \Psi _{X h} > \  
      <  \Psi _{Y h} , g > \big | dX dY.$$
Then,  Schur Lemma yields,
      $$  | < Af , g> |^2  \leq  (2 \pi h)^{- 2n}  \Vert \psi_F \Vert _ { L^1( \R^{2n}) } ^2
       \int _{\R^{2n} } | < f , \Psi _{X h} > |^2  dX
         \ \ \int _{\R^{2n} } | <  \Psi _{Y h} , g >|^2  dY. $$
 This implies, using (\ref{integ}), that,
        $$  | < Af , g> | \leq \Vert F \Vert _ {D_h(\R^{2n})  } \ 
        \Vert f \Vert_{L^2(\R^n)} \  \Vert g \Vert_{L^2(\R^n)} . $$

$ii)$ We now prove that the Wick symbol  of the operator $A$ defined in  (\ref{def-A-Wick}) is equal to $F$. 
        
        First, the Berezin symbol  $ \sigma_h^{\rm ext}A(U , V)$ satisfies,
        $$ (\sigma_h^{\rm ext}A ) (U , V) =  \frac  { < A \Psi _{U h},  \Psi _{V h} >} { <  \Psi _{U h},  \Psi _{V h} >}. $$
        Next, from (\ref{def-A-Wick}),
        $$  (\sigma_h^{\rm ext}A ) (U , V) = (2 \pi h)^{-2n} \int _{\R^{4n} }  F^{\rm ext}(Y, X) 
        \frac { < \Pi _{X h}  \Pi _{Y h} \Psi _{U h} , \Psi _{V h} >   } 
     { < \Psi _{U h},  \Psi _{V h} >}   dX dY.$$
     In view of  (\ref{PSEC-1}), 
$$
        (2 \pi h)^{-2n} \frac { < \Pi _{X h}  \Pi _{Y h} \Psi _{U h} , \Psi _{V h} >   } 
     { < \Psi _{U h},  \Psi _{V h} >}   = 
     e^{ \frac  {1 }    {2h } ( < Y , X > _{\rm C} - < U , V>_{\rm C})  }  
     K_h (X , V) \ K_h ^{\star} (Y , U),$$
     where,
     $$  K_h ( X , V)=  (2 \pi h)^{-n} 
      e^{ - \frac  {1 }    {2h } (  |X|^2 +
       < X , V> _{\rm C} )} 
    ,\quad  
        K_h ^{\star} (Y , U)= (2 \pi h)^{-n} e^{ - \frac  {1 }    {2h } (  |Y|^2 
        + < U , Y> _{\rm C})}.
         $$
    Also, from   Theorem  \ref{NR} $i)$,
       $$  \int _{\R^{2n} }  F^{\rm ext}(Y, X) 
        e^{ \frac  {1 }    {2h }  < Y , X > _{\rm C}  } 
         K_h ^{\star} (Y , U) dY = 
         F^{\rm ext}(U, X) e^{ \frac  {1 }    {2h }  < U , X > _{\rm C}  } $$
    and with Theorem  \ref{NR} $ii)$,
          $$  \int _{\R^{2n} }  F^{\rm ext}(U, X) 
            e^{ \frac  {1 }    {2h } ( < U , X > _{\rm C}  }
             K_h ( X , V) dX =  F^{\rm ext}(U, V) 
              e^{ \frac  {1 }    {2h } ( < U , V > _{\rm C}  }.$$
Consequently, the  Berezin symbol of the operator  $A$ is thus  
         $F^{\rm ext}(X , Y)$. By diagonal restriction, 
         the  Wick symbol  of $A$ is equal to $F$.

$iii)$ We have,
     $$ < {\rm Op}_h^{\rm Wick}(F) \Psi _{U + Z/2,  h}, \Psi _{U - Z/2,  h}> 
     $$
     $$ =
  (2 \pi h)^{-2n} \int _{\R^{4n} }  F^{\rm ext}(Y, X) 
   < \Psi _{U + Z/2 , h } , \Psi _{Y , h } > 
   < \Psi _{Y , h } ,  \Psi _{X , h } > 
   < \Psi _{X , h }, \Psi _{U - Z/2,  h}>  dX dY. $$
 In view of (\ref{PSEC}),
      \begin{align}\nonumber |& < {\rm  Op}_h^{\rm Wick}(F) \Psi _{U + Z/2,  h}, \Psi _{U - Z/2,  h}> |\\[3mm] 
& \leq 
  (2 \pi h)^{-2n} \int _{\R^{4n} }  |F^{\rm ext}(Y, X) |
   e^{-   \frac {1}  {4h} |U + Z/2 - Y   | ^2  } 
    e^{-   \frac {1}  {4h} |Y-X| ^2  } 
    e^{-   \frac {1}  {4h} |U - Z/2 - X   | ^2  } dX dY \\[3mm] 
    & \leq (2 \pi h)^{-2n} \int _{\R^{4n} } |\psi_F (X-Y) |
     e^{-   \frac {1}  { 2h} |U  - \frac {X+Y}  {2}  | ^2 
      -  \frac {1}  { 8h}  | X-Y-Z  |^2   } dX dY.  
            \end{align}
  Therefore, 
     $$ d_h (  {\rm Op}_h^{\rm Wick}(F)  ) (Z) \leq 
     (2 \pi h)^{-n} \int _{\R^{2n} } |\psi_F (W) | 
      e^{- \frac {1}  { 8h}  | Z - W  |^2   } dW.$$
      
$iv)$ This point comes from the fact that, 
    $$  W_h (T)^{\star} \Pi_{X,h} \Pi_{Y,h}  W_h (T) =  
    \Pi_{X -T,h} \Pi_{Y-T,h} .$$
   This equality is a consequence of the results recalled in  Section 7.1.
   
    \fpr

      \subsection{Composition. Proofs of Theorems 
      \ref{comp-abs}, \ref{comp-Wick} and  \ref{Miz}. } \label{s3.4}

   { \it Proof of Theorem \ref{comp-abs}.}  
   
   $i)$ We have,
   $$ d_h (A \circ B) (V) =  \sup _{X-Y = V} 
     | < (A \circ B )\Psi _{X h},  \Psi _{Y h} > |. $$ 
     From (\ref{integ}), we notice that, 
   $$ d_h (A \circ B) (V) \leq    (2 \pi h) ^ {-n} 
   \sup _{X-Y = V} 
      \int _{\R^{2n}}  | < A \Psi _{X h},  \Psi _{Z h} > | 
      | < B \Psi _{Z h},  \Psi _{Y h} > | dZ. $$
     Clearly, from (\ref{confin}),
      $$ | < A \Psi _{X h},  \Psi _{Z h} >|  \leq  d_hA (X-Z)  ,\quad    
      | < B \Psi _{Z h},  \Psi _{Y h} >|  \leq  d_hB (Z-Y).  $$    
Thus,
\begin{align}\nonumber
  d_h (A \circ B) (V)& \leq    (2 \pi h) ^ {-n} 
   \sup _{X-Y = V} 
      \int _{\R^{2n}}  d_hA (X-Z)  d_hB (Z-Y)  dZ \\[3mm] 
& \leq   (2 \pi h) ^ {-n} \int _{\R^{2n}}  
     ( d_h A) (V-U)  ( d_h B) (U) dU. \end{align}
    Point $ i)$  then follows.

   {\it ii)}     For all $F$ and  $G$ in $D_{ h} ( \R^{2n})$, 
 the operators
    $A_h =  {\rm Op}_h^{\rm Wick}(F)$ and $B_h =  {\rm Op}_h^{\rm Wick}(G)$ 
    belong to $C_h$.   From $i)$, $A_h \circ B_h$ also lies in  
  $C_h$. We have already seen that the Wick symbol 
    $\sigma _h^{\rm Wick}(A_h \circ B_h)\in D_{ h} ( \R^{2n})$. 
 It is denoted by $F \sharp _h^{\rm Wick} G$.  Concerning identity  (\ref{sharp-Wick}), the convergence of the integral comes from the fact that,    
   $$ |  F  ^{\rm ext} ( X+T, X)   G  ^{\rm ext} ( X , X+T) | \leq 
     e^{  \frac {1}   {2h} |T|^2 }  \psi_F(T) \psi_G(T) .$$

If $A_h =  {\rm Op}_h^{\rm Wick}(F) $ and  $B_h =  {\rm Op}_h^{\rm Wick}(G) $ then we have,

\begin{align}\nonumber \sigma _h^{\rm Wick} (A_h \circ  B_h) &= < (A_h \circ  B_h) \Psi_{X h}, 
   \Psi_{X h} >  \\[3mm] 
   &  = (2 \pi h)^{-n}
     \int _{ \R^ {2n}}  < B_h \Psi_{X h} , \Psi_{Y h} > 
      < A_h \Psi_{Y h} , \Psi_{X h} > dY  \\[3mm] 
      & = (2 \pi h)^{-n}
     \int _{ \R^ {2n}} F^{\rm ext} (Y , X)  G^{\rm ext} (X , Y)
     | <  \Psi_{X h} , \Psi_{Y h} > |^2 dY \\[3mm] 
     & = (2 \pi h)^{-n}
     \int _{ \R^ {2n}} F^{\rm ext} (Y , X)  G^{\rm ext} (X , Y)
     e^{ - \frac {1}   {2h} |X-Y|^2 }  dY \\[3mm] 
     &  = (2 \pi h)^{-n}
     \int _{ \R^ {2n}}  
   F  ^{\rm ext} ( X+T, X)   G  ^{\rm ext} ( X , X+T)   
   e^{ - \frac {1}   {2h} |T|^2 } dT.\end{align}
   Theorem \ref{comp-abs} is thus proven.

   \fpr

  {\it Proof of Theorem \ref{comp-Wick}.}

$i)$ Set $\lambda $ and $\mu$ satisfying $0 < \mu < \lambda $. 
          Let  $F$ and $G$ belonging to $D_{ \lambda h}(\R^{2 n })$, $\widetilde F$ and 
 $\widetilde G$ being their holomorphic extensions (of Theorem \ref{equiv-Dh}), $F^{\rm ext}$ and $G^{\rm ext}$ 
their sesqui-holomorphic extensions. We apply Theorem 
 \ref{puiss-ordre2} with 
 $L^{\rm Wick}$ defined (\ref{L-Wick}). We then obtain,
  $$ \frac {h^m}  {m!} \Vert (L^{\rm Wick})^m F \Vert_{ D_{\mu  h }( \R^{2n})} 
  \leq C \frac {C^m}  {(\lambda - \mu) ^m}
  \Vert  F \Vert_{ D_{\lambda  h }( \R^{2n})}. $$
Point i) then follows.

   {\it ii)}  According to (\ref{sharp-Wick}), $F \sharp_h^{ Wick } G$ is the restriction to the diagonal of the following function,
     $$   \Psi (X , Y)  = (2 \pi h)^{-n}
     \int _{ \R^ {2n}}  
   F  ^{\rm ext} ( X+T, X)   G  ^{\rm ext} ( Y , Y+T)   
   e^{ - \frac {1}   {2h} |T|^2 } dT. $$ 
    From (\ref{F-ext}), we write with $Z\in\R^{2n}$,
    $$  F  ^{\rm ext} ( X+T, X) =  \widetilde  F \Big  ( x +  \frac { Z}  {2}   , \xi 
     -i \frac {Z} {2}    \Big ) $$
     $$ G  ^{\rm ext} ( Y , Y+T)  =  \widetilde  G \Big  ( y + 
      \frac {\overline Z}  {2}   , \eta 
     +i \frac {\overline Z} {2}    \Big ).$$
Consequently,
   $$  \Psi (X , Y) = (2 \pi h)^{-n}
     \int _{ \C^ {n}}  
     \widetilde  F \Big  ( x +  \frac { Z}  {2}   , \xi 
     -i \frac {Z} {2}    \Big )
      \widetilde  G \Big  ( y +  \frac {\overline Z}  {2}   , \eta 
     +i \frac {\overline Z} {2}    \Big )
 e^{ - \frac {1}   {2h} |Z|^2 } dZ.$$
For all $\Phi\in {\cal S}(\R^{2n} )$, 
     set
       $$ T_h \Phi =  (2 \pi h)^{-n}  \int _{ \R^ {2n}} 
       \Phi (Z)  e^{ - \frac {1}   {2h} |Z|^2 } dZ. $$ 
     Then, we have, 
       $$  \frac {d} {dh } T_h \Phi =  \frac {1} {2} 
       T_h (\Delta \Phi )   ,\quad
       T_0 \Phi = \Phi.$$
   Thus, the function   
       $ F \sharp_h^{ Wick } G$ is the restriction to the diagonal  $X=Y$ of the function $T_h \Phi $ where, 
   $$ \Phi  (X , Y , Z) =  
        \widetilde  F \Big  ( x +  \frac { Z}  {2}   , \xi 
     -i \frac {Z} {2}    \Big )
      \widetilde  G \Big  ( y +  \frac { \overline Z}  {2}   , \eta 
     +i \frac {\overline Z} {2}    \Big ) .$$
     We notice that,
     $$ \Delta _Z \Phi  (X , Y , Z) = 
     (\partial _x -i \partial _{\xi}) \cdot (\partial _y + i \partial _{\eta})
     \Phi  (X , Y , Z). $$
As a consequence, the function   $(F \sharp_h^{ Wick } G)$ is the restriction to the diagonal  $X=Y$ of the function $e^{h L^{\rm Wick} } (F \otimes G)$. 
The proof is completed.
   
   \fpr

    {\it   Proof of Theorem \ref{Miz}.  }  
    
    From   Theorem  \ref{Dh-Gev} $ii)$, there exists  $C$ satisfying the following inequality, 
  for every  $F$ and $G$ in $D_{\lambda h } (\R^{2n})$,
         
   $$       
    \frac  {h^{|\alpha|}}   { 2^{|\alpha|} \alpha! }  
       | \partial _X^{\alpha}   F  ( X)|  \ 
         |\partial _{\overline X}^{\alpha}  G (X)| 
         \leq C \frac  {(C)^{|\alpha|}} {\lambda^{|\alpha|}}
         \Vert F \Vert _{D_{\lambda h } (\R^{2n})} \ 
       \Vert G \Vert _{D_{\lambda h } (\R^{2n})} . $$
Then, the series in right hand side of (\ref{comp-Miz}) converges absolutely if   $\lambda $ is sufficiently large (and $h\leq 2$). The sum of the series is    the restriction to the diagonal  $X=Y$  of $e^{h L^{\rm Wick} }(F\otimes G)$,  which proves the result.

    \fpr

     \section{Action of some operators.  Proofs of Theorems \ref{Dh-Gev} $ii)$ and \ref{holo-convol}.}\label{s4}
     
     This section is mainly concerned with operators (differential and convolution) mapping $D_{\lambda } ( \R^{2n})$ into $D_{\mu } ( \R^{2n})$, $\mu<\lambda$. We also prove Theorem \ref{Dh-Gev} $ii)$.
    
      \subsection{Action of differential  operators. } \label{s4.1}

       Theorem \ref{Cauchy} below  together with the fact that every function in
     $D_{\mu} ( \R^{2n})$  is bounded show  that  
     $D_{\lambda } ( \R^{2n})$ is included in the  Gevrey class of order  $\frac{1}{2}$.     
    
  \begin{theo}\label{Cauchy} For any  $\lambda$ and $\mu$ satisfying
  $ 0 < \mu  < \lambda$, for every     multi-index $ \alpha $, 
  the operator  $\partial ^{\alpha} $ is bounded from
  $D_{\lambda } ( \R^{2n})$ into  $D_{\mu} ( \R^{2n})$ and we have, 
  $$  \Vert \partial ^{\alpha} F \Vert _{D_{\mu}( \R^{2n})} 
  \leq ( {\lambda  - \mu} ) ^{- 2n}  \sqrt {\alpha ! }
   \Big (\frac  {C }  {\lambda  - \mu} \Big )^{\frac{| \alpha|}{ 2}} 
   \Vert  F \Vert _{D_{\lambda}( \R^{2n}) }, $$
 for some constant real number $C$.
  
    \end{theo}
  
  One can find a similar action in   scales of Banach algebras in \cite{Ovs}. 
  The proof involves the following result.
   
   \begin{theo}\label{Herm-est}  Set, for $\theta>0$, $\alpha\in\N$ and $x\in\R$,
  $$ H_{\alpha}(x,\theta  ) =   e^{ \frac {\theta   } { 2 } |x |^2  } 
   \partial ^ {\alpha  }  e^{-  \theta  |x |^2  } .$$
 Then,
   \be\label{Hermite} \Vert  H_{\alpha} ( \cdot ,\theta  )\Vert _{L^1(\R^n)}
   \leq ( C \sqrt \theta )^{|\alpha | } \theta ^{-n/2} \sqrt {\alpha ! }. \ee

  \end{theo}

  {\it Proof of Theorem \ref{Herm-est}. } 
  
  Using notations in 
  \cite{A-D-SM-T}, one sets,
  $$ W _{1/2} (\rho_m) = \int_{\R} e^{ - \frac {x^2}  {2} }|H_m(x)| dx ,$$
 where $H_m $ are the Hermite polynomial functions,
  $$ H_m(x)= (-1)^m  e^{   x ^2  } \frac {d^m } { dx^m }  e^{ - x ^2  }. $$
  
 From Theorem 2.1 in \cite{A-D-SM-T}, with $p=1/2$, we  get as $m\rightarrow +\infty$,
   $$  W _{1/2} (\rho_m) = C ( 2/e)^{m/2} (m+1) ^{(m+1)/2} (1+ o(1)).$$

   That is, setting  for all integers $\alpha \geq 0$,
   $$  G_{\alpha} (x) =  e^{ \frac {1 } { 2 } x ^2  } 
    d ^ {\alpha  }  e^{-  x ^2  },$$
   there exists $C>0$ satisfying,
  $$ \Vert  G_{\alpha} \Vert _ {L^1(\R)} \leq C   (\alpha+1) ^{(\alpha +1)/2}.  $$
  Using Stirling formula, we obtain with another  constant $C>0$ the estimate,
  $$ \Vert  G_{\alpha} \Vert _ {L^1(\R)} \leq C^{\alpha} \sqrt {\alpha ! }.   $$
     This implies Theorem \ref{Herm-est}.
  
  \fpr

  {\it Proof  of Theorem  \ref{Cauchy}.} 
  
  Fix $F$ in  $D_{\lambda}( \R^{2n})$ and let
  $ F^{\rm ext}$ be its sesqui-holomorphic extension. Each harmonic function is left invariant by the heat operator action. Then, for each $ \theta >0$, one sees that, 
  $$  F^{\rm ext}(X , Y) =  \frac {\theta ^{2n}}  { \pi ^{2n}} 
   \int _ { \R^{4n}} e^{- \theta ( |X-X'|^2 + |Y-Y'|^2) } 
   F^{\rm ext}(X' , Y') dX' dY'.$$
 The parameter $\theta $ will be chosen thereafter.  Thus,
 $$  \partial _X^{\alpha} \partial _{\overline Y}^{\beta} F^{\rm ext}(X , Y)= 
  \frac {\theta ^{2n}}  { \pi ^{2n}} 
   \int _ { \R^{4n}}  F^{\rm ext}(X' , Y') H_{\alpha} ( X- X', \theta ) 
      H_{\beta} ( Y- Y', \theta ) 
    e^{- \frac {\theta}  {2}   ( |X-X'|^2 + |Y-Y'|^2) } dX' dY' .$$
  Since $F\in D_{\lambda}( \R^{2n})$,  we have if  $0< \mu < \lambda$,
    $$ e^{- \frac {1}  {4 \mu} |X-Y|^2} 
    |  \partial _X^{\alpha} \partial _{\overline Y}^{\beta} F^{\rm ext}(X , Y) | 
    \leq  \frac {\theta ^{2n}}  { \pi ^{2n}}  
     \int _ { \R^{4n}}  e^{\varphi (X , X' , Y , Y', \theta)}
    |H_{\alpha} ( X- X', \theta ) 
      H_{\beta} ( Y- Y', \theta ) | dX' dY', $$
      where
      $$ \varphi (X , X' , Y , Y', \theta) = 
      - \frac {1}  {4 \mu} |X-Y|^2 +  
       \frac {1}  {4 \lambda} |X'-Y'|^2 
       - \frac {\theta}  {2}   ( |X-X'|^2 + |Y-Y'|^2) .$$
       Setting $ \theta = 1/ (\lambda - \mu)$, we obtain that,
       $ \varphi (X , X' , Y , Y', \theta) \leq 0$ for all 
       $(X , X' , Y , Y')$. Therefore,
       $$  e^{- \frac {1}  {4 \mu} |X-Y|^2} 
    |  \partial _X^{\alpha} \partial _{\overline Y}^{\beta} F^{\rm ext}(X , Y) | 
    \leq  \psi_H (X - Y),$$
    with,
    $$ \psi_H (X - Y) =   \frac {\theta ^{n}}  { \pi ^n} 
     \int _ { \R^{4n}}\psi_F (X '- Y') |H _{\alpha} ( X- X', \theta ) 
      H_{\beta} ( Y- Y', \theta ) | dX' dY'. $$
 We then deduce the estimate,
       $$ \Vert  \psi_H  \Vert _{L^1 ( \R^  {2n}) } \ \leq 
       \Vert  \psi_F  \Vert _{L^1 ( \R^  {2n}) } \ 
        \Vert H_{\alpha} ( \cdot , \theta )  \Vert _{L^1 ( \R^  {2n}) } \ 
         \Vert H_{\beta} ( \cdot , \theta )  \Vert _{L^1 ( \R^  {2n} ) }. $$
   Thus, using  (\ref{Hermite}),   Theorem  \ref{Cauchy}  is completed.

         \fpr

       {\it Proof of Theorem \ref{Dh-Gev} ii).} 
       
       We apply the above Theorem with $\mu = \lambda /2$. Noticing that        $D_{\mu}( \R^{2n})$ is included in  $L^{\infty}( \R^{2n})$, 
     we obtain, 
  $$  \Vert \partial ^{\alpha} F \Vert _{L^{\infty}( \R^{2n})} 
  \leq ( {\lambda /2} ) ^{- n}  \sqrt {\alpha ! }
   \Big (\frac  { 2C }  {\lambda } \Big )^{\frac{| \alpha|}{ 2}} 
   \Vert  F \Vert _{D_{\lambda}( \R^{2n}) }.$$
This implies the stated result. 
  
  \fpr 
  
  \begin{theo}\label{puiss-ordre2} Let  $L$ be a homogeneous differential operator of order 2 with constant coefficients. Then, there exists $C>0$ satisfying if 
 $0 < \mu < \lambda $ and  $\lambda - \mu >C$, for all
    $h>0$ and  $F\in D_{\lambda h }( \R^{2n})$, 
  \be\label{Lm} \frac {h^m}  {m!} \Vert L^m F \Vert_{ D_{\mu  h }( \R^{2n})} 
  \leq  \frac {C^{m+1}}  {(\lambda - \mu) ^m}
  \Vert  F \Vert_{ D_{\lambda  h }( \R^{2n})} .\ee
  
  \end{theo}

  {\it Proof of Theorem \ref{puiss-ordre2}. } 
  
  There exists $C>0$ such that,
  \begin{align}\nonumber
  \Vert L^m F \Vert_{ D_{\mu  h }( \R^{2n})} &\leq C^m 
  \sup _{|\alpha | = 2m} 
  \Vert \partial ^{\alpha} F \Vert _{D_{\mu h}( \R^{2n})}  \\[3mm]
  & \leq  K   \sqrt {(2m)!}  
  \frac  {C ^m}  {h^m ( \lambda - \mu)^m } 
  \Vert F \Vert_{ D_{\lambda   h }( \R^{2n}}. 
  \end{align}
Then,
  $$ \frac {h^m}  {m!} \Vert L^m F \Vert_{ D_{\mu  h }( \R^{2n})} 
   \leq K   \frac { \sqrt {(2m)!} }  {m!} \ 
   \frac  {C ^m}  {( \lambda - \mu)^m } 
    \Vert F \Vert_{ D_{\lambda   h }( \R^{2n})}.$$
One knowns that,
  $$  \frac { \sqrt {(2m)!} }  {m!}  \leq 2^m. $$
These points proves inequality (\ref{Lm}). 
  
  \fpr

   \subsection{Convolution operators action.} \label{s4.2}

  {\it Proof of Theorem \ref{holo-convol}. }       
  
   Set $ (\nu_h)\in {\cal M} ( \gamma , \R^{2n} )$  
  with $\gamma <0$.  Then,
  $$ \widehat \nu_h (u) = e^{ h(\varphi_1(u) + \varphi_2(u) ) } ,$$
where  $\varphi_1(D)$ is bounded in $L^{\infty}(\R^{2n})$ and  
  $\varphi_2$ is satisfying (\ref{majo-der-phi})(\ref{majo-phi}). 
Let, 
       $$   \Phi_h^{(2)}  (X) = \int _{\R^{2n}}
        e^{iX \cdot u + h \varphi ^{(2)}   (u)}
      du .$$ 
We can write,
      $$ (1 + |X|^2) ^{n+1} \  \Phi _h ^{(2)} (X) =
       \sum _{\alpha} c_{\alpha}  \int _{\R^{2n}} 
         e^{iX \cdot u }  \partial _u ^{\alpha}
          e^{ h \varphi ^{(2)} (u)}   du.  $$
      We remark that the  function $\Phi _h ^{(2)}$   has a holomorphic extension verifying, 
      $$ \Phi_h ^{(2)} (X + i Y) = (1 + |X+ i Y|^2) ^{-n-1}
       \sum _{\alpha}   c_{\alpha}  \int _{\R^{2n}} 
        e^{i (X+iY) \cdot u }  \partial _u ^{\alpha}
          e^{ h \varphi ^{(2)} (u)}   du. $$
    If $ \delta < 4 | \gamma | $ then we have,
       $$  \int _{\R^{2n}}  e^{- \frac {1}  {\delta h}   | Y| ^2}  
       \Big  |  e^{i (X+iY) \cdot u }  \partial _u ^{\alpha}
          e^{ h \varphi  ^{(2)} (u)}  \Big  | du \leq C. $$
    Therefore,  
     \be\label{majo-Phih-2} \int _{\R^{2n}}   |\Phi_h  ^{(2)} (X+ i Y)|  dX \leq C 
      e^{\frac {1}  {\delta h}   | Y| ^2} ,\ee 
      if $ \delta < 4 | \gamma | $.   
    The density   $ \Phi_h$ of $\nu_h$ with respect to the Lebesgue measure satisfies,      
      $$  \Phi_h (X) =  e^{ h \varphi_1  (D)} \Phi ^{ (2)}(X) .$$
   If  $\Phi ^{ (2)}$ verifies  (\ref{majo-Phih-2}) and if 
       $e^{ h \varphi_1  (D)}$ is a bounded convolution operator in 
       $L^ {\infty}(\R^{2n})$ then $ e^{ h \varphi_1  (D)} \Phi ^{ (2)} $ 
      also has a holomorphic extension satisfying  (\ref{majo-Phih-2})
        according to  of Theorem 
       \ref{equiv-Dh} $vi)$, 
       $$ \int _{\R^{2n}}   |\Phi_h  (X+ i Y)|  dX \leq C 
      e^{\frac {1}  {\delta h}   | Y| ^2} . $$ 
  Then, the convolution by    $\nu_h$ is mapping $L^ {\infty}(\R^{2n})$ 
    into the space of functions  $G$ which have  a holomorphic extension 
    $\widetilde G$ verifying,
       $$ \int _{\R^{2n}}   |\widetilde G  (X+ i Y)|  dX \leq C 
      e^{\frac {1}  {\delta h}   | Y| ^2}. $$ 
In view of  Theorem \ref{equiv-Dh} $v)$, this space is included in        $D _{\varepsilon h} (\R^{2n})$ if  $\varepsilon < \delta$. Therefore, 
   the  convolution by  $\nu_h$ maps  $L^ {\infty}(\R^{2n})$ 
      into $D _{\varepsilon h} (\R^{2n})$  if $ \varepsilon  < 4 | \gamma | $.
      Theorem \ref{holo-convol} is complete. 
      
      \fpr
      
See also \cite{B-H-P} for  other examples of operators acting in scales of Banach spaces.

   \section{Composition of symbols. Proofs of Theorems \ref{expo}, \ref{compo-gene} and \ref{quantif}.  } \label{s5}

   {\it Proof of Theorem \ref{expo}.  }
   
    We can assume that $m>0$. Let $ \lambda > m$ and $\rho_h $ be  the measure satisfying,
 $$\widehat {\rho_h }(u) = e^{h ( \Phi (u) - \lambda |u|^2 ) }. $$
We have $(\rho_h )\in {\cal M} ( \gamma , \R^{4n} )$ 
  with $\gamma = m - \lambda <0$. Let $H_{\lambda  h}$  be the operator in Theorem \ref{equiv-Dh}. Set,
$$ TF =  ( H_{\lambda  h} F ) \star \rho_h .$$
In view of Theorem 
\ref{holo-convol}, for any $G\in L^{\infty} (\R^{4n})$,
$$ \Vert G \star \rho_h \Vert _ { D_{\mu h}(\R^{4n})} 
\leq C \Vert G\Vert _ {L^{\infty} (\R^{4n})} ,$$
if  $ \mu < 4 (\lambda - m)$.
According to Theorem \ref{equiv-Dh},
$$ \Vert  H_{\lambda  h}F \Vert _ {L^{\infty} (\R^{4n})} \leq C 
\Vert  F \Vert _ { D_{\lambda h }( \R^{4n})}. $$
Thus, for all $F\in D_{\lambda h }( \R^{4n})$,
$$ \Vert   TF \Vert _ { D_{\mu h}(\R^{4n})} \leq 
C \Vert  F \Vert _ { D_{\lambda h }( \R^{2n}) (\R^{4n})}. $$
The operator $ H_{\lambda  h}$ has a symbol given by $e^{\lambda h |u|^2}$. 
The symbol of the  convolution operator by  $\rho_h$ is
$e^{h ( \Phi ^{\nu}(u) - \lambda  |u|^2)}$. Therefore,
the symbol of $T$ is equal to $e^{h  \Phi ^{\nu}(u)} $.
Theorem \ref{expo} is proven.

   \fpr

\begin{prop}\label{restric} Fix $\rho > 2$ and  $h>0$. Then, 
the diagonal restriction operator $R$ 
maps  $D_{\rho  h}(\R^{4n})$ 
into  $D_{ h}(\R^{2n})$.

\end{prop}

{\it Proof of Proposition \ref{restric}. } 

Take $F\in D_{\rho  h}(\R^{4n})$. Let
$F^{\rm ext}$ be the sesqui-holomorphic extension of $F$ (that is,  $F^{\rm ext} (X_1 , Y_1 , X_2 , Y_2 )$ is holomorphic in the variables $X_1$ and $X_2$, and 
anti-holomorphic in $Y_1$ and $Y_2$) which satisfies, for each $(X_1 , Y_1 , X_2 , Y_2 ) \in \R^{8n}$,
$$ | F^{\rm ext} (X_1 , Y_1 , X_2 , Y_2 ) | \leq C 
e^{\frac  { 1 }   {4 \rho h  } ( | X_1 - Y_1|^2 +  | X_2 - Y_2|^2 )  }. $$
The sesqui-holomorphic extension of  $(RF) (X)= F(X , X)$ is 
$ (RF) ^{\rm ext}(X, Y) =  F^{\rm ext} (X , Y , X , Y)$. We have,
$$  | (RF) ^{\rm ext}(X, Y)  | \leq  C e^{\frac  { 1 }   {2 \rho h  }
  | X - Y|^2   } =  C e^{\frac  { 1 }   {4  h  }
  | X - Y|^2   } \psi (X-Y),$$
where $\psi (X) = e^{ ( \frac{1}{2 \rho h} - \frac{1}{4h})  | X |^2   } $. 
The function $\psi$ belongs to $L^1(\R^{2n}) \cap L^ { \infty}(\R^{2n})$. 
Thus, Proposition \ref{restric} is proven. 

\fpr

  {\it Proof  of Theorem \ref{compo-gene}}. 
  
   $ i)$ Since $F$ and $G$ are in  $ D_{\lambda h}(\R^{2n})$ then 
  $F\otimes G$ is clearly in   $ D_{\lambda h}(\R^{4n})$. Also, 
  $\lambda > m + \frac{1}{2}$ implies that there exists $\rho$ satisfying 
  $2 < \rho < 4 (\lambda -m)$. In view of  Theorem
  \ref{expo} and using the abuse of  notations as agreed below this result, the operator  $e^{h \Phi ^{\nu} (D) }$ is bounded from
  $D_{\lambda h } (\R^{4n})$ into $D_{\rho h} (\R^{4n})$. Since $\rho >2$, 
Proposition \ref{restric} shows that the operator  $R$ is bounded from $D_{\rho h} (\R^{4n})$ 
into  $D_{ h} (\R^{2n})$, which proves  point i).

$ii)$ Using the abuse of notations mentioned below    Theorem 
      \ref{holo-convol} and Theorem  \ref{expo}, we have,  
  \begin{align}\nonumber
     e^{h \varphi (D)}  R   e^{h L ^{\nu} (D) } (F \otimes G) &=
       R  e^{h\varphi (D_1 + D_2)   + h   L ^{\nu }(D) }
          ( F \otimes G)\\[3mm]
          &= R  e^{ h  L^{\rm Wick} }  ( e^{h \varphi (D)} F 
          \otimes  e^{h \varphi (D)} G).   \end{align}
     
From Theorem \ref{comp-Wick},  
 $$   R  e^{h L^{\rm Wick} } (\Phi \otimes \Psi  )  
 = \Phi  \sharp_h^{ Wick } \Psi. $$
 Therefore, 
          $$  e^{h \varphi (D)}  R   e^{h L ^{\nu} (D) } (F \otimes G) 
           = 
          ( e^{h \varphi (D)} F) \sharp _h^{\rm Wick}
           ( e^{h \varphi (D)} G) .$$
        Thus, if  $A_h = {\rm Op}_h^{\nu} (F)$ then we also have,
           $A_h=  {\rm Op}_h^{\rm Wick} ( e^{h \varphi (D)} F)$ 
           and similarly replacing $A_h$ and $F$ with $B_h$ and $G$. Consequently,
           $$  {\rm Op}_h^{\rm Wick} ( e^{h \varphi (D)} ( F \sharp _h ^{\nu} G ))
           = A_h \circ B_h.$$
          This is equivalent to,
           $$  {\rm Op}_h^{\nu} (  e^{h \varphi (D)}  R   e^{h L^{\nu} (D) } (F \otimes G)  ) = A_h \circ B_h.$$
\fpr

     {\it Proof  of Theorem \ref{quantif}. }
    
$i)$ If $F=1$ then 
    $ F \star \nu = 1$ and thus, ${\rm Op}_h^{\rm Wick}( F \star \nu) = 
    {\rm Op}_h^{\nu} (F) = I$.

$ii)$ Using Theorem \ref{compo-gene}, we have,
  $$ F \sharp _h ^{\nu }  G =  R ( I + h L^{\nu }(D) )
  (F \otimes G) + {\cal O } (h^2). $$
Consequently,
  $$ F \sharp _h ^{\nu }  G  - G \sharp _h ^{\nu }  F  =
  h R   L^{\nu } (D)
  (F \otimes G - G \otimes F  ) + {\cal O } (h^2) .$$
 We note that,
     $$ R \ L^{\rm Wick}  (F  \otimes G - G  \otimes F ) = \frac {1}  {i} \{ F , G \} $$
    where  $\{ F , G \}$ stands for the  Poisson bracket,
     $$ \{ F , G \}  = \partial _{\xi }F \cdot \partial _{x } G -  
      \partial _{x }F \cdot \partial _{\xi  } G .$$
  We then deduce that,
      $$  R \ L^{\nu}  (D)(F  \otimes G - G  \otimes F ) = \frac {1}  {i} \{ F , G \}. $$
 Theorem \ref{compo-gene} is  proven.

    \fpr

    \section{Examples.  } \label{s6}

   \subsection{Weyl calculus. } \label{s5.1}
   
     The operator    $  {\rm Op }_h^{\rm Weyl}  (F) $ is commonly defined for every  symbol $F\in C^ {\infty} (\R^{2n})$ being 
     bounded  together with its derivatives and for all $f\in {\cal S}(\R^n)$ by, for $x\in\R^n$,
 $$    ( {\rm Op }_h^{\rm Weyl}  F)f (x) = (2\pi h) ^ {-n}\int _{ \R^ {2n}} 
 e^{ \frac {i } {h} (x- y) \cdot \tau } F \left (\frac {x+y}  {2}  , \tau   \right  ) f(y) dy d\tau.$$

   The following proposition  is well known.

   \begin{prop}\label{mes-Weyl}   For all symbols $F\in {\cal S}(\R^{2n})$,
  \begin{align}\label{ex-Uh-AW-2} 
  <( {\rm Op }_h^{\rm Weyl}  F) \Psi _{X h}, 
 \Psi _{X h}> &=  (\pi  h) ^{-n} \int _{\R^{2n}} 
 e^{ - \frac {1} {h} |X-T|^2 } F(T) dT \\[3mm]
 & = (F \star \nu_h^{\rm Weyl} ) (X), 
 \end{align}
where $ \nu_h^{\rm Weyl}$ denotes the measure defined in (\ref{mesure-Weyl}), that is,
   $$ d \nu_h^{\rm Weyl} (X) =  ( \pi h) ^{-n}
  e^{ - \frac {1} {h} |X|^2 } dX. $$  
\end{prop}

   \begin{theo}\label{Ho-W-Comp} For all $F$ and $G$ in  
   $ {\cal S} ( \R^{2n})$ and any  $h$ sufficiently small, one has, 
\be\label{Compo-We} (F \sharp _h ^{\rm Weyl} G) (X) = R
 ( e^{  i (h/2)   \sigma (D_X , D_Y) } (F \otimes G ) 
  ) (X). \ee

   \end{theo} 
   
 If we consider the equality only as a formal series, Theorem \ref{Ho-W-Comp} is actually a classic result  in \cite{HO} (Section 18.5) but the exponential in (\ref{Compo-We}) is given only as formel series in \cite{HO}.

    {\it  Proof  of Theorem  \ref{Ho-W-Comp}. }
    
     Let $\nu_h =  \nu_h^{\rm Weyl}$ be the measure defined in (\ref{mesure-Weyl}). Then, we have,
    $$ \widehat {\nu_h }  (\xi) = e^{ - \frac  { 1}  { 4h} |\xi|^2  }. $$
  Consequently,
    $$ (F \star  \nu_h^{\rm Weyl}) = e^{ h \varphi (D) } F, $$
  where $\varphi$ is written as  (\ref{phi-D}) 
  with $A =  I/2$. In view of Theorem 
    \ref{comp-symb-quad}, setting $A = I/2$, we get 
    $$  (F \sharp _h ^{\rm Weyl} G) (X) = R e^{hL^{\nu }(D) } ( F \otimes G) $$
  with
    $$L^{\nu } = \frac  {1}  {2} 
     \begin{pmatrix} 0 & B +  I/2   \\ B^T +   I/2  &  0  \end{pmatrix} 
,\quad  B  = - \frac  {1}  {2}  \begin{pmatrix} 1 & i \\ - i &   1 \end{pmatrix} . $$  
Thus, we obtain,
    $$ L^{\nu } = \frac  {1}  {4}  \begin{pmatrix} 0 &0 & 0 & -i \\ 
    0&0 &  i &0 \\ 
    0 &  i &0&0 \\ 
   - i & 0 & 0 & 0 \end{pmatrix} . $$  
 Therefore, $ L^{\nu }(D)   =  (i /2 ) \sigma ( (D_{x_1},D_{\xi_1}) , 
   (D_{x_2},D_{\xi_2}) ) $.  This proves Theorem \ref{Ho-W-Comp} .

     \fpr

          \subsection{Anti-Wick calculus. } \label{s6.2}

We remind below the standard definition of anti-Wick calculus.

\begin{defi}  Take $F\in L^{\infty} (\R^ { 2n})$.
The operator  $A =  {\rm Op}_h ^ {AW} (F)$ is defined by,  
   \be\label{def-AW}  < Af , g>  =  (2 \pi h) ^{-n}  \int _{ \R^{2n}} < f , \Psi_ {X h} > F(X)  <  \Psi_ {X h}, g > dX, \ee
   for all $f$ and $g$ in $ L^2(\R^n)$.
  \end{defi}
 
  
In particular, we see that,
   $$ \Vert A \Vert _{{\cal L} ( L^2(\R^n))   } \leq 
    \Vert F\Vert _{ L^{\infty} (\R^{2n})   } .  $$ 
   
Let us prove  that Definition (\ref{def-AW}) implies,
\be\label{ThAf} (T_hAf) (X) =  \int _{ \R^{2n}} e^{ \frac {1}  {2h}<Y , X>_{\rm C}  } 
   F(Y)  (T_hf) (Y) d\mu_h (Y),\ee
   where 
   $$  d\mu_h (Y)=  (2 \pi h) ^{-n}    e^{ - \frac {1}  {2h} |Y|^2 } dY .$$
Indeed,
$$ (T_hf) (Y)=  e^{ \frac { 1 }  { 4h }|Y|^2  } < f , \Psi_ { Y h }>  $$ 
    and
    \begin{align}
\nonumber (T_hAf) (X) &= e^{ \frac { 1 }  { 4h }|X|^2  }  < Af , \Psi_ { X h }>    \\[3mm]
\nonumber   & =  (2 \pi h) ^{-n}  \ e^{ \frac { 1 }  { 4h }|X|^2  } 
   \int _{ \R^{2n}}
    < f , \Psi_ {Y h} > F(Y)  <  \Psi_ {Y h},  \Psi_ { X h } > dY \\[3mm]
\nonumber    &=  (2 \pi h) ^{-n}  \int _{ \R^{2n}} 
    e^{ \frac { 1 }  { 4h }(|X|^2 - |Y|^2) } F(Y) 
     <  \Psi_ {Y h},  \Psi_ { X h } > (T_hf) (Y) dY.    \end{align}
   From (\ref{PSEC-1}),
$$  <  \Psi_{X,h} ,  \Psi_{Y,h} > = 
e^{ - \frac {1}  {4h} ( |X|^2 + |Y|^2) 
+ \frac {1}  {2h} <X , Y>_{\rm C}}. $$ 
Then equality  (\ref{ThAf}) holds true. 
   
Also note that (\ref{ThAf}) can be written as, 
   $$ T_h Af = {\cal P}_h ( M_F ) T_h F, $$
where ${\cal P}_h $ stands for the  orthogonal projection in $L^2 (  \R^{2n}, \mu_h )$ 
on the subspace of    anti-holomorphic functions.
   
This type of operators is often called   Toeplitz operator. 
    See \cite{CHS-2,Cob-4}. 
       Proposition \ref{mes-AW} below is standard.

 \begin{prop}\label{mes-AW}    For any symbol $F\in {\cal S}(\R^{2n})$,
 \be\label{ex-Uh-AW} <( {\rm Op }_h^{\rm AW}  F) \Psi _{X h}, 
 \Psi _{X h}> =  (2 \pi  h) ^{-n} \int _{\R^{2n}} 
  e^{ - \frac {1} {2h} |X-T|^2 } F(T) dT  \ee 
 $$ 
 = (F \star \nu_h^{\rm AW}) (X), $$ 
where $\nu_h^{\rm AW}$ is the measure defined  in (\ref{mesure-AW}).
   
  \end{prop}    
 
  \begin{theo} \label{Comp-AW} For every  $F$ and $G$ in  
   ${\cal S} ( \R^{2n})$ and if $h$ is small enough, we have,
  $$ (F \sharp _h ^{\rm AW} G) (x , \xi) = 
  \big ( e^{  (h/2)   (D_x+ i D_{\xi}) \cdot  (D_ y - i D_{\eta})  }
   (F \otimes G )  \big ) (x , \xi , x ,  \xi ) .$$

   \end{theo}

   {\it Proof  of Theorem \ref{Comp-AW}. }
   
     If  $\nu_h =  \nu_h^{\rm AW}$   defined in (\ref{mesure-AW}) then,    
    $$ \widehat {\nu_h }  (\xi) = e^{ - \frac  { 1}  { 4h} |\xi|^2  }.  $$
Consequently,
    $$ F \star  \nu_h^{\rm Weyl} = e^{ h \varphi (D) } F, $$
 where $\varphi$ is given in   (\ref{phi-D}) 
   with $A =  I$. 
One then applies Theorem \ref{comp-symb-quad}  
    with $A= I$.  Then, the matrix $L^{\nu}$  in (\ref{Lnu}) is written as,
   $$ L^{\nu}  =  \frac {1}  {4}  \begin{pmatrix} 0 & 0& 1& - i \\ 
     0& 0& i & 1 \\
     1 & i &  0 & 0 \\ 
     -i & 1 &0 &0 \end{pmatrix} .$$
One then deduces Theorem \ref{Comp-AW}.   
     
     \fpr

 \subsection{Classical calculus. }   \label{s6.3}

 The classical (or left, or standard) quantization is the  operator    $  {\rm Op }_h^{\rm class}  (F) $  defined by, for $x\in\R^{n}$,
 $$    ( {\rm Op }_h^{\rm class}  F)u (x) = (2\pi h) ^ {-n}\int _{ \R^ {2n}} 
 e^{ \frac {i } {h} (x- y) \cdot \tau } F (x , \tau) u(y) dy d\tau. $$

    \begin{prop}\label{mes-class}  The Wick symbol of the operator  
     $  {\rm Op }_h^{\rm class}  (F) $ 
is given by,  for $X\in\R^{2n}$,
   \be\label{ex-Uh-class}  \sigma_h^{\rm Wick} (  {\rm Op }_h^{\rm class}  (F))(X) 
   =   2 ^{-n/2}  ( \pi h) ^{-n} \int _{ \R^ {2n}} F(T) 
    e^{ - \frac {1} {2h} (|x- t|^2 + |\xi -\tau  |^2 + 2i (x-t) \cdot 
    (\xi -\tau ) )} dt d\tau .
     \ee  
 
 \end{prop}

 {\it  Proof  of Proposition  \ref{mes-class}. }
 
  First,
 $$  \sigma _h^{\rm Wick}(  {\rm Op }_h^{\rm class}  F) (Z) = 
  (2\pi h) ^ {-n}\int _{ \R^ {3n}}  e^{ \frac {i } {h} (x- y) \cdot \xi }
   F (x , \xi )  \Psi_{Z,h} (y) \ \overline { \Psi_{Z,h} (x)}  dx dy d \xi .$$
Next,  using (\ref{coh-state}), with $Z=(z , \zeta)\in\R^{2n}$,
 $$  \sigma _h^{\rm Wick}(  {\rm Op }_h^{\rm class}  F) (z , \zeta) =   (  \pi h)^{ -n /2}
 ( 2 \pi h)^{ -n }  \int _{ \R^ {3n}}  e^{ \frac {i } {h} (x- y) \cdot ( \xi - \zeta) }
 F(x , \xi) \  e^{ - \frac {1} {2h} ( | x - z|^2 + | y - z|^2) } dx dy d\xi.$$
Then, we also note that, 
   $$   \int _{ \R^ {n}}  e^{ \frac {i } {h}y\cdot ( \zeta - \xi) - \frac {1} {2h}  | y - z|^2  }
   dy = (2 \pi h)^ {n/2} \   e^{ \frac {i } {h} z \cdot (\zeta - \xi) -  \frac {1 } {2h} | \zeta - \xi|^2 }. $$
Consequently,
   $$  \sigma _h^{\rm Wick}(   {\rm Op }_h^{\rm class}  F) (Z) = (2 )^ { - n/2}
   ( \pi h)^ { - n}  \int _{ \R^ {2n}} e^{ - \frac {1 } {2h} Q(Z-X) } F(X) dX, $$
   with,
   $$ Q(x , \xi ) =  |x|^2 + |\xi |^2   +  2i x\cdot \xi. $$
This ends the proof of Proposition \ref{mes-class}.

   \fpr

   \begin{theo}\label{Comp-class} For any  $F$ and $G$ lying in  
   $D_{\lambda h} (\R^{2n})$ (with $\lambda $ large enough), 
    we have, 
  $$ (F \sharp _h ^{\rm class} G) (x , \xi) =
   \big ( e^{ - i h D_x \cdot D_ {\eta}) } (F \otimes G ) 
 \big ) (x , \xi , x , \xi) .$$

   \end{theo}

   {\it Proof  of Theorem \ref{Comp-class}. }
   
     If  $\nu_h$ equals to   the measure $ \nu_h^{\rm class}$ defined in (\ref{mesure-class}) then we have,
    $$ \widehat {\nu_h }  (\xi) = e^{ - \frac  { 1}  { 4h} |\xi|^2  }.  $$
Therefore,
    $$ F \star  \nu_h^{\rm class} = e^{ h \varphi (D) } F $$
where $\varphi$ is written as in  (\ref{phi-D}) with,
    $$  A  = \frac {1}   {2}  
      \begin{pmatrix} 1 & - i \\ - i &   1 \end{pmatrix}.  $$
According to Theorem \ref{comp-symb-quad},   
     the matrix   $L^{\nu}$ is expressed as, 
   $$ L^{\nu}  = \frac  {1}   {2}   \begin{pmatrix} 0 & 0& 0& - i \\ 
     0& 0& 0 & 0 \\
     0 & 0  &  0 & 0 \\ 
    - i & 0 &0 &0 \end{pmatrix}. $$
 We then deduce Theorem \ref{Comp-class}.

     \fpr

     \subsection{Example:  L\'evy processes.  } \label{s6.4} 
     
 We prove Theorem \ref{grano-salis} and give some more details between the general quantization formula and random operators.    
       
         {\it Proof of Theorem \ref{grano-salis}. }  
         
     Let   $\nu_t$ be the measure  satisfying 
    (\ref{Four-mes-proc}). 
   We set,
   $$  \varphi_1(u)  = \int _{|y| > \varepsilon} e^{i u \cdot y } d\rho (y)  $$
where   $\varepsilon <1$ will be fixed below and we write,
  $$  \varphi_2(u)  =  i m \cdot u  -   \frac {1}  {2} ( au) \cdot u - 
  \int _{|y| > \varepsilon} \left ( 1 + i \frac   {u \cdot y } {1 + |y|^2}
   \right )  d\rho (y)  $$
  $$  +   \int _{|y| < \varepsilon} 
  \Big ( e^{i u \cdot y } -1 - i \frac   {u \cdot y } {1 + |y|^2} \Big ) 
     d \rho (y), $$
   $ i)$ With the function  $\varphi_1$,  we can assign a convolution operator  by the bounded measure   $\rho (y) \chi _{ |y| >  \varepsilon}$. This operator is bounded  
in $L^{\infty} (\R^n)$ and its symbol is equal to $\varphi_1 (u)$.

  $ ii)$   We have from Taylor formula with integral remainder,
     $$  e^{i u \cdot y } -1 - i u \cdot y = 
     \int _0^1 (1-\theta)  e^{i\theta u \cdot y } ( u \cdot y )^2 d\theta. $$
We see that,
  $$ | \partial ^{\alpha }\varphi_2 (u) | \leq A_{\alpha } |u|^2 
   + B_{\alpha } .$$
    Also, 
     $$  \frac {1}  {2} (au ) \cdot  u \geq \delta |u|^2.$$
 Thus, 
    if  $\varepsilon $ is chosen sufficiently small then the family of measures
     $(\nu_t)$ belongs to the set  ${\cal M} (  - \gamma   , \R^{n} )$.
          
 The existence of the density  $\Phi_t$ and the existence of the holomorphic extension of $\Phi_t$ together with  its estimates are a consequence of Theorem \ref{holo-convol}.

 \fpr 
 
  \begin{theo} \label{rand-op} Let  $(X_t)$ be a L\'evy process on $\R^{2n}$ with a definite positive diffusion matrix. The processus $(X_t)$ is considered as function $X(t , \omega)$ where $\omega $ belongs to a probability space $(\Omega,\mu_\Omega)$. For all  $t>0$, set   the Weyl translation operator $V(t , \omega)$
associated with $ X(t , \omega)$. 
Let  $A_0 =  {\rm Op}_h^{\rm Wick}(F)$ be the operator associated with a function  $F\in D_h( \R^{2n})$ as in Section \ref{Wick calculus}. For every  $t>0$, set
 $$ A(t , \omega) =  V(t , \omega)^{\star}  A_0  V(t , \omega).$$
Then, we have the identity,
$$  {\rm Op}_t^{\nu_t}(F) = \int _{\Omega}  A(t , \omega) d \mu_\Omega(\omega). $$

   \end{theo}

 {\it Proof of Theorem \ref{rand-op}.}
 
  From  of Theorem  \ref{bij} $iv)$, we have,
 $$  V(t , \omega)^{\star} {\rm Op}_t^{\rm Wick}(F) V(t , \omega) = 
 {\rm Op}_h^{\rm Wick}(F_{t , \omega})  
,\quad  F_{t , \omega}(X) = F(X +  X(t , \omega) ).$$
  Using  the transfer Theorem,
  $$ \int _{\Omega} F(X +  X(t , \omega) ) d\omega = 
  ( F \star \nu_t ) (X) .$$
Therefore,
  $$ \int _{\Omega}  A(t , \omega) d \omega  = 
  {\rm Op}_t^{\rm Wick}(F \star \nu_t )$$
which proves Theorem \ref{rand-op}.

  \fpr

          \subsection{Born-Jordan quantization. } \label{s5.5}

  First,   for every  $\lambda$ in $[0, 1]$, the intermediate operator 
 $  {\rm Op }_h^{\lambda}  (F) $ is  given by, $t\in\R^n$,
 $$    ( {\rm Op }_ h^{\lambda}  F)u (t) = (2\pi h) ^ {-n}\int _{ \R^ {2n}} 
 e^{ \frac {i } {h} (t-s) \cdot \tau } F ( (1- \lambda )s  + \lambda t  , \tau) u(s) ds d\tau. $$
Then,  the  Born-Jordan operator  is defined by (see \cite{d-Gos}),
 \be\label{def-BJ}     {\rm Op }_h^{\rm BJ}  F = \int _0^1  ( {\rm Op }_ h^{\lambda}  F)\, d \lambda .\ee 

      \begin{prop}\label{BJandW} We have,
    
      $$  \sigma _h^{\rm Wick} ( {\rm Op }_h^{\rm BJ}  F) (Z) = 
       \int _{ \R^ {2n}} 
  F( Z+T)  \Phi (T) dT, $$
  for all $Z\in\R^{2n}$
 where  $\Phi $ is defined in  (\ref{mesure-BJ}).

      \end{prop}

{\it Proof  of Proposition \ref{BJandW}.}  

The  Wick symbol  of the operator 
     $  {\rm Op }_h^{\rm BJ}  (F) $ 
  is,
 $$  \sigma _h^{\rm Wick}(   {\rm Op }_h^{\rm BJ}  F) (z , \zeta) = 
 $$
  $$  (2\pi h) ^ {-n}  \int _{ \R^ {3n} \times [0, 1]} 
   e^{ \frac {i } {h} (x- y) \cdot \xi }
   F ((1-\lambda) x + \lambda y , \xi )  \Psi_{Z,h} (y) \ \overline { \Psi_{Z,h} (x)}  dx dy d \xi d\lambda. $$
    From (\ref{coh-state}), we get,
 $$  \sigma _h^{\rm Wick}(   {\rm Op }_h^{\rm BJ}  F) (z , \zeta)   
 $$
$$  = (  \pi h)^{ -n /2} ( 2 \pi h)^{ -n }   \int _{ \R^ {3n}\times [0, 1]}  
    e^{ \frac {i } {h} (x- y) \cdot ( \xi - \zeta) }
     F((1-\lambda) x + \lambda y  , \xi) 
 \  e^{ - \frac {1} {2h} ( | x - z|^2 + | y - z|^2) } dx dy d\xi  d\lambda.   $$
Then, we use the change of variables,
 $$ (1-\lambda) x + \lambda y  = z + t 
 ,\quad \xi = \zeta + \tau $$
 $$ x-y = u  ,\quad \lambda = \frac {1 } {2} + \sigma. $$
Thus,
 $$  \sigma _h^{\rm Wick}  ( {\rm Op }_h^{\rm BJ}  F)(z , \zeta) = 
  \int _{ \R^ {2n}} 
  F( z + t , \zeta + \tau)  \Phi (t , \tau ) dt d \tau $$
  with,
\begin{align} 
\nonumber \Phi (t , \tau ) &=  (  \pi h)^{ -n /2}  ( 2 \pi h)^{ -n }  
   \int _{ \R^ {n}\times [- (1/2), 1/2]} 
    e^{ \frac {i } {h} u \cdot \tau  } \ 
      e^{ - \frac {1} {h} ( | t + \sigma  u|^2 
    - \frac {1} {4h} | u|^2  } 
       du d\sigma \\[3mm]
  \nonumber   &   =  (  \pi h)^{ -n } 
      \int _{-1/2} ^{1/2}
      (1 + 4 \sigma^2)^{ -n /2} 
       \exp \left (  - \frac { 1}   { h (1 + 4 \sigma^2) } 
       ( |t|^2 +    |\tau |^2  + 4i \sigma t \cdot \tau )  \right )  
       d \sigma .\end{align}
         The proof of Proposition  \ref{BJandW}   is completed.

     \fpr

      \section{Usual operators.} \label{s7}
      
   We recall in  Section \ref{s7} some standard notions with some of their  properties which are essentially  known.       
       \subsection{Coherent states.}  \label{s6.1}
      
      The $\Psi_{X h}$ (coherent states)
     are a family of functions of $L^2(\R^{n})$,
     indexed by $X = (x , \xi) \in \R^{2n} $ and depending on $h>0$, defined for all        $X =(a,b) \in \R ^{2n}$ and  $h>0$ by,
\be\label{coh-state}   \Psi_{X,h} (u)  =    (\pi h)^{ -n/4}
e^{-\frac{| u-x|^2 }{ 2h}} e^{\frac{i}{ h} u .\xi - \frac{i}{ 2h} x. \xi },
  \quad u\in \R^n  .\ee 
  (see \cite{C-R}).
 These functions are used, for example, in \cite{A-J-N-2,Fol,LER,Per,U-2}. One knows that, 
 $$  \Psi_{X,h} = W_h(X)    \Psi_{0,h}$$
with $ W_h(X) $ denoting  the  Weyl translation operators defined by,  
 writing $X = (x , \xi)$, 
 \be\label{trans-Weyl}  ( W_h(X)  f) (u) = f(u - x) 
  e^{ \frac { i}  { h} u \cdot \xi 
 -  \frac { i}  {2 h} x \cdot \xi }. \ee 
One also knows that, 
 \be\label{trans-Weyl-comp}  W_h (X) W_h (Y) = e^{ \frac { i}  { 2h} \sigma ( X , Y) }  W_h (X+Y),   \ee 
with the symplectic form $\sigma$ given by
 $\sigma ( (x , \xi) , (y , \eta))= y \cdot \xi - x \cdot \eta$. 
 One then has,
 $$  W_h (X)  \Psi_{Y,h} = e^{ \frac { i}  { 2h} \sigma ( X , Y) } 
  \Psi_{X+Y,h} .$$
In particular,
\be\label{PSEC}  <  \Psi_{X,h} ,  \Psi_{Y,h} > = 
e^{ - \frac {1}  {4h}  |X - Y|^2 + \frac {i}  {2h} 
\sigma (X , Y) }, \ee 
and equivalently,
\be\label{PSEC-1}  <  \Psi_{X,h} ,  \Psi_{Y,h} > = 
e^{ - \frac {1}  {4h} ( |X|^2 + |Y|^2) 
+ \frac {1}  {2h} <X , Y>_{\rm C}}. \ee

Besides,  for all $f$ and $g$ in $L^2(\R^{2n})$, 
\be\label{integ} < f , g> = (2\pi h)^{-n} \int _{\R^{2n}} 
  < f ,  \Psi_{X,h} > \ < \Psi_{X,h}  , g> dX,  \ee
and then,
  \be\label{integ-2} (2 \pi h) ^{-n} \int _{ \R^{2n}} | < f ,  \Psi_{X,h} > |^2 dX  = 
  \Vert f \Vert ^2_{L^2(\R^n)}. \ee

   \subsection{Reproducing kernel. }  
   
   The following facts concerning reproducing kernels are standard. 
   
  \begin{theo}\label{NR} i) For each $F$ on $\R^{2n}$ satisfying,
    $$ |F(Y) | \leq K e^{a |Y|^2 } ,\quad 
    ( \partial  _{y_j } + i \partial _{\eta_j }) F = 0 
,\quad  1 \leq j \leq n, $$
 where $a>0$, and for any $h$ sufficiently small, one has,
    $$ (2 \pi h)^{-n}
     \int _{\R^{2n} } F(Y)  e^{  \frac {1}  {2h}  
     < X , Y>_{\rm C} -  \frac {1}  {2h} |Y|^2} dY   
      = F(X).$$
      ii) For every function $F$ on  $\R^{2n}$ verifying, 
       $$ |F(Y) | \leq K e^{a |Y|^2 },\quad
    ( \partial  _{y_j } - i \partial _{\eta_j }) F = 0 
,\quad 1 \leq j \leq n, $$
     one has,
     $$ (2 \pi h)^{-n}
     \int _{\R^{2n} } F(Y)  e^{  \frac {1}  {2h}  
     < Y , X>_{\rm C}  
     -  \frac {1}  {2h} |Y|^2} dY   
      = F(X).$$

   \end{theo}

  \subsection{Bargmann transform. } 
  
 One can additionally use  the Bargmann transform of a function $f$ in $L^2(\R^n)$ in order to express the action of some given operator. This transform is a function on  $\R^{2n}$ given by,
\be\label{Barg} (T_h f) (X) = \frac { < f , \Psi_{X,h} > } { < \Psi_{0,h} , \Psi_{X,h} > }  \ee 
$$ = e^{ \frac {1}  {4h}  |X|^2 }  < f , \Psi_{X,h} > .$$
 
This integral transform is a particular case of the transforms presented in 
 \cite{H-J},  \cite{Z} (Chapter 13), \cite{SAM,Trev}. 
The transform used here is also employed in \cite{D-G}  and implicitly in \cite{BL-JFA}.  These transforms are known under different terminology such as the FBI transform.

 The Bargmann transform is a partial isometry
  from  $L^2(\R^n)$ into a subspace of   $L^2(\R^{2 n }, \mu_h)$.
  Note that in contrast  to some   other  transforms using the same terminology, $T_h$
here is satisfying,
      $$ (\partial _{x_j } - i \partial _{\xi_j })(T_hf) (X) = 0.$$

From (\ref{coh-state}), one has,
$$ (Thf) (x , \xi) =  (\pi h)^{ -n/4} \int _ {\R^n} f(u) 
e^{-\frac{| u|^2 }{ 2h}} \ e^{-\frac{1 }{ h} u \cdot (x- i  \xi) 
- \frac{1 }{ 4h} (x- i  \xi)^2  } du .$$ 
One then gets  the anti-holomorphic property. One also has, 
$$ \int _ {\R^{2 n}} |(Thf) (x , \xi) |^2 d \mu_h (x , \xi) = 
 (2 \pi h) ^{-n} \int _ {\R^{2 n}} | < f , \Psi_{X,h} > |^2 dX = 
 \Vert f \Vert _{L^2 (\R^{n})} ^2. $$
 
One can therefore equivalently write operators in 
 $L^2 (\R^n)$  using the Berezin symbols and the Bargmann transform functions. This is an alternative representation of the one used above.

 \begin{prop} \label{altern} 
 We have the following properties.
 
 i) For every bounded  operator  $A$ in $L^2(\R^n)$
 and for any   $f\in {\cal S}(\R^n)$,  
  $$ ( T_h Af) (X) = \int _{\R^{2n} }  \  {\cal P}_h (Y , X) \ 
   (\sigma_h ^{\rm ext}A) (Y, X)  (T_hf) (Y)  d\mu_h (Y) $$
where
   $$ {\cal P}_h (Y , X) = e^{ \frac {1}  {2h}
  < Y , X>_{\rm C}  } . $$
We denoted by $\sigma_h ^{\rm ext}A$  the Berezin symbol of $A$, 
   by $T_hf$  the  Bargmann transform of $f$ defined  by (\ref{Barg}), 
   and we set
  $$  d\mu_h (Y) = (2 \pi h)^{-n} e^{ - \frac {1}  {2h} | Y | ^2  } dY .$$
   
   ii) If a function   $\Psi (X , Y)$ 
  satisfies  $(\partial  _{x_j } + i \partial _{\xi_j }) \Psi  = 0$ 
and 
   \be\label{ecriture-A}  ( T_h Af) (X) = \int _{\R^{2n} }  \  {\cal P}_h (Y , X) \ 
  \Psi  (Y, X)  (T_hf) (Y)  d\mu_h (Y), \ee 
for all $f\in {\cal S}(\R^n)$, then
 $\Psi =  \sigma_h^{\rm ext} A $.

    \end{prop}

      {\it Proof  of Proposition \ref{altern}}. 
      
      $i)$ From (\ref{Barg}), we get,
       $$  ( T_h Af) (X) = e^{ \frac {1}  {4h} |X|^2 }  < Af ,  \Psi_{X h} >. $$
 Using (\ref{integ}),
     \begin{align}  ( T_h Af) (X)\nonumber &= e^{ \frac {1}  {4h} |X|^2 } (2 \pi h)^{-n} 
        \int _ {  \R^ {2n } }  < f ,  \Psi_{Y h} >  
        < A \Psi_{Y h} , \Psi_{X h} > dY \\[3mm]
       \nonumber   & =   (2 \pi h)^{-n}   \int _ {  \R^ {2n } } 
        e^{ \frac {1}  {4h} (|X|^2 - |Y|^2 )   } 
         ( T_h f) (Y) \  (\sigma_h ^{\rm ext}A) (Y, X)  <  \Psi_{Y h} , \Psi_{X h} > dY.    \end{align} 
      From (\ref{PSEC-1}),
       $$  <  \Psi_{Y h} , \Psi_{X h} >= 
        e^{ - \frac {1}  {4h} (|X|^2 + |Y|^2 ) 
        + \frac {1}  {2h} < Y , X>_{\rm C}  } .
        $$

  $ii)$ If  $A$ is defined by  (\ref{ecriture-A}), we deduce that,  
        for all  $U$ and  $V$ in   $\R^{2n}$,
        $$ \frac  { < A \Psi _{U h},  \Psi _{V h} >} { <  \Psi _{U h},  \Psi _{V h} >} =  \int _{\R^{2n} } K_h (U , V , Y) \ 
          \Psi (Y , V)  d\mu_h (Y), $$
   with
   $$  K_h (U , V , Y)  = \frac { e^{ - \frac {1}  { 4h} V|^2 } 
   (T_h \Psi _{U h}) (Y) \   {\cal P}_h (Y , V)}
    { <  \Psi _{U h},  \Psi _{V h} >} . $$
  We then observe that,
    $$  K_h (U , V , Y)  =   e^{  \frac {1}  {2h} 
    ( < U , Y > _{\rm C} + <Y , V> _{\rm C} - <U , V>_{\rm C})}. $$
According to Theorem \ref{NR}, one has for every function $G$ satisfying 
    $(\partial  _{y_j } + i \partial _{\eta_j }) G = 0$, 
    $$ \int _{\R^{2n} } G(Y)  e^{  \frac {1}  {2h}  < U , Y > _{\rm C} } 
     d\mu_h (Y)  = G(U). $$
   Thus, using
     $(\partial  _{y_j } + i \partial _{\eta_j }) \Psi (Y , V)  = 0$, 
    we obtain,
     $$ \int _{\R^{2n} } K_h (U , V , Y) \ 
          \Psi (Y , V)  d\mu_h (Y)=  \Psi (U , V).$$
         Therefore,  the Berezin symbol of $A$ is indeed satisfying
          $  \sigma_h^{\rm ext} A  =  \Psi$.

       \fpr

     \section{Semi-groups of measures and bounded operators. } 
   
   Proposition \ref{act-proc-op} below shows that every semi-group of  probability measures 
 on the classical phase space $\R^{2n}$ has an  action on ${\cal L} ( L^2( \R^n))$ and also on the subalgebra $C_h$ of 
     operators $A$ having a  confinement function $d_h(A)$ belonging to
    $L^1 ( \R^{2n})$.

 \begin{prop}\label{act-proc-op}  Let $A$ be a bounded   operator  in 
      $L^2( \R^n)$. Fix  a   probability measure $\nu$  on $\R^{2n}$. 
We define  an operator  $\Pi(\nu) (A) $ by, 
    $$ \Pi(\nu) (A)= \int _{ \R^{2n}} W_h (-X) A W_h (X) d\nu (X), $$
   where $ W_h (X)$ are the Weyl translation operators in (\ref{trans-Weyl}). Then, the following three points hold true.

    i) The map $ \Pi(\nu)$ 
    is bounded from ${\cal L} ( L^2( \R^n))$ into itself and from $C_h$ 
  into itself.
    
    ii) We have,
    $$ \Pi (\mu) \Pi (\nu) = \Pi (\mu \star \nu).$$
   iii) We also have,
    $$ \Pi (\mu) (A \circ B) =  \Pi (\mu) (A) \circ   \Pi (\mu) (B).$$
      
      \end{prop}
    
   The proof is straightforward.   Note that the proof of  point ii) comes from (\ref{trans-Weyl-comp}).

laurent.amour@univ-reims.fr\newline
{\sc Laboratoire de Math\'ematiques de Reims, UMR CNRS 9008,\\ Universit\'e de Reims Champagne-Ardenne
 Moulin de la Housse, BP 1039,
 51687 REIMS Cedex 2, France.}
 
 \vskip 0.4cm
 
richard.lascar@univ-cotedazur.fr\newline
{\sc Laboratoire Jean Alexandre Dieudonn\'e, UMR 7351, Universit\'e C\^ote d'Azur, Parc Valrose, 06200, Nice, France.}

 \vskip 0.4cm

jean.nourrigat@univ-reims.fr {\sl and }  jean.nourrigat0675@orange.fr\newline
{\sc Laboratoire de Math\'ematiques de Reims, UMR CNRS 9008,\\ Universit\'e de Reims Champagne-Ardenne
 Moulin de la Housse, BP 1039,
 51687 REIMS Cedex 2, France.}

     \end{document}